\documentclass[smallextended,referee,envcountsect,]{svjour3}
\smartqed

\journalname{Arxiv}
\usepackage{amsmath,amsbsy,amssymb,amscd,mathrsfs,graphicx}
\usepackage{tikz}
\usepackage{xcolor,color}
\usepackage{floatrow,float}
\usepackage[titletoc, title]{appendix}
\usepackage[colorlinks,linkcolor=blue,citecolor=blue]{hyperref}
\usepackage{extarrows}
\usepackage{pgfplots}
\usepackage{diagbox}
\usepackage{booktabs}
\usepackage{multirow}
\usepackage[capitalise, nosort]{cleveref}
\usepackage{algorithm}
\usepackage{algorithmic}

\usetikzlibrary{arrows.meta, positioning}
\tikzset{elegant/.style={smooth,thick,samples=50,black}}
\tikzset{eaxis/.style={->,>=stealth}}

\newtheorem{coro}{Corollary}[section]

\newtheorem{lem}{Lemma}[section]

\newtheorem{thm}{Theorem}[section]

\newtheorem{assum}{Assumption}

\newcommand{\proj}[0]{ {\bf proj}}

\newcommand{\conv}[1]{{\bf conv}\left( {#1} \right)}

\newcommand{\R}{\,{\mathbb R}}

\newcommand{\argmin}[0]{ {\mathop{{\rm  argmin}}\,}}

\newcommand{\reff}[1]{$\rm \color{blue}(\ref{#1} )$}
\newcommand{\ssnm}[1]
{
	\left\vert\kern-0.25ex
	\left\vert\kern-0.25ex
	\left\vert
	{#1}
	\right\vert\kern-0.25ex
	\right\vert\kern-0.25ex
	\right\vert
}

\crefname{equation}{}{}
\crefname{lem}{Lemma}{Lemmas}
\crefname{thm}{Theorem}{Theorems}
\crefname{assum}{Assumption}{Assumptions}
\makeatletter
\def\spher@harm#1{%
	\vbox{\hbox{%
			\offinterlineskip
			\valign{&\hb@xt@2\p@{\hss$##$\hss}\vskip.2ex\cr#1\crcr}%
		}\vskip-.36ex}%
}
\def\gshone{\spher@harm{.}}
\def\gshtwo{\spher@harm{.&.}}
\def\gshthree{\spher@harm{.&.&.}}
\let\gsh\spher@harm
\makeatother

\newcolumntype{I}{!{\vrule width 1,5pt}}
\newlength\savedwidth

\newlength\savewidth

\newcounter{mnote}
\let\oldmarginpar\marginpar
\renewcommand\marginpar[1]
{\-\oldmarginpar[\raggedleft\footnotesize #1]{\raggedright\footnotesize #1}}

\usepackage[T1]{fontenc} 
\usepackage[utf8]{inputenc} 
\usepackage{authblk}
\usepackage[all]{xy}
\begin{document}

\title{Multiobjective Accelerated Gradient-like Flow with Asymptotic Vanishing Normalized  Gradient}


\author{Yingdong Yin}

\institute{Yingdong Yin\at
National Center for Applied Mathematics in Chongqing, Chongqing Normal University, Chongqing, 401331, China\\
             yydyyds@sina.com
}

\date{Received: date / Accepted: date}

\maketitle

\begin{abstract}
{ For unconstrained convex smooth multiobjective optimization, we propose a novel gradient-like flow that incorporates the asymptotic vanishing normalized gradient. In the scalar case, this flow reduces to a first-order method with strong empirical performance, as introduced by Wang et al. [SIAM J. Sci. Comput., 2021].} We prove the existence of a trajectory solution and, using Lyapunov analysis, establish convergence rates of $O(1/t^2)$ and $O(\ln^2 t / t^2)$ under two distinct parameter choices. Under certain assumptions, we further show that the trajectory converges to a weakly Pareto optimal solution. By discretizing the flow, we derive a new multiobjective accelerated gradient method that achieves a convergence rate of $O(\ln^2 k / k^2)$. { Numerical experiments demonstrate that both our continuous flow and discrete algorithm lead to faster convergence on most problems.}
\end{abstract}

\section{Introduction}
This paper considers $\mathbb R^n$, an $n$-dimensional Euclidean space with inner product $\langle \cdot, \cdot \rangle$ and induced norm $\| \cdot \|$. We study the unconstrained multiobjective optimization problem \reff{eq:MOP}:
\begin{equation}\label{eq:MOP}
\min_{x \in \mathbb{R}^n} F(x) := \begin{pmatrix} f_1(x) , \cdots  f_m(x) \end{pmatrix}^\top, \tag{MOP}
\end{equation}
where each $f_i: \mathbb R^n \to \mathbb R$ ($i=1,\ldots,m$) is convex and continuously differentiable. Unlike in single-objective optimization, \reff{eq:MOP} involves multiple, typically conflicting objectives. This necessitates different optimality concepts and presents challenges for both theoretical analysis and numerical methods.

{ For this reason, accelerated first-order methods and gradient flows have been extensively studied in single-objective optimization \cite{attouch2022damped,attouch2018fast,attouch2019fast,attouch2024convex,luo2023accelerated,luo2022differential,su2016differential}, while research in multiobjective optimization remains limited. In recent years, Sonntag and Peitz have conducted pioneering work in this area \cite{sonntag2024fastgradientflow,sonntag2024fastNestrovAlgorithm}. Below we mainly review their recent work and elaborate on our motivations.}

 \subsection{Accelerated gradient methods and flows for multiobjective optimization}
 The classical approach to solving the \cref{eq:MOP} problem is scalarization, which often requires properly chosen weights and is not easily capable of approximating the Pareto front. To address this issue, Filge et al. \cite{fliege2000steepest} proposed the multiobjective steepest descent method, where the descent direction is defined as  
 \begin{equation}\label{eq:direction}
 d(x) = \underset{d \in \mathbb{R}^n}{\argmin} \left\{ \frac{1}{2} \|d\|^2 + \max_{i=1,\cdots,m} \Big\langle \nabla f_j(x), d \Big\rangle \right\}.
 \end{equation}
 Moreover, some studies \cite{attouch2015multiibjective}  showed that the steepest descent direction in \cref{eq:direction} is  the projection of the zero vector onto the convex hull \( C(x) = \textbf{conv}\{\nabla f_j(x) : j= 1,\ldots,   m\} \), i.e., \( d(x) = -\textbf{proj}_{C(x)}(0) \). This leads to the following multiobjective gradient flow
 \begin{equation}\label{eq:MOG}
 \dot{x}(t) + \proj_{C(x(t))}(0) = 0, \tag{MCSD}
 \end{equation}
 called the \textit{Multiobjective Continuous Steepest Descent} as shown in \cite{Attouch2014}. { This dynamic system is also related to the proximal gradient method \cite{tanabe2019proximal}.} 
 
 In \cite{attouch2015multiibjective}, Attouch et al. introduced the \textit{Inertial Multiobjective Gradient System}, we can write it as follows 
 \begin{equation}\label{eq:IMOG}
 \ddot{x}(t) + \alpha(t) \dot{x}(t) + \proj_{C(x(t))}(0) = 0,\tag{IMOG}
 \end{equation}
 which is the continuous time limit of the heavy-ball method if $\alpha (t)=\alpha $ is a constant.  
 However, the convergence of \cref{eq:IMOG} requires that the damping coefficient has a positive lower bound to ensure, which is not {satisfied} if $\alpha(t) = \frac{\alpha}{t}$.
 Therefore, associating \cref{eq:IMOG} with multiobjective accelerated method, such as the  proximal gradient case \cite{tanabe2023accelerated}, is challenging for convergence analysis.
 
 Recently, to address above issue, Sonntag and Peitz  \cite{sonntag2024fastgradientflow,sonntag2024fastNestrovAlgorithm} introduced the \textit{Multiobjective Inertial Gradient-like Dynamical System with Asymptotic Vanishing Damping} as follows
 \begin{equation}\label{eq:MAVD}
 \ddot{x}(t) + \frac{\alpha}{t} \dot{x}(t) + \proj_{C(x(t))}(-\ddot{x}(t)) = 0, \tag{$\text{MAVD}$}
 \end{equation}
 which has an equivalent presentation:
 \begin{equation*}
 \frac{\alpha }{t}\dot x(t)+\proj_{C(x(t))+\ddot x(t)}(0)=0.
 \end{equation*}
 Furthermore, discretizing \cref{eq:MAVD} yields the multiobjective accelerated gradient method with iterations
 \begin{equation}\label{eq:AccG}
 \left\{ \begin{aligned}
 y_k &= x_k + \frac{k-1}{k+2} (x_k - x_{k-1}), \\
 x_{k+1} &= y_k - s \sum_{i=1}^m \theta_{i,k} \nabla f_i(y_k).
 \end{aligned} \right. \tag{AccG}
 \end{equation}
 where \(\theta_k = (\theta_{i,k})_{i=1}^m \in \argmin_{\theta \in \Delta^m} \Big\| s \Big( \sum_{i=1}^m \theta_i \nabla f_i(y_k) \Big) - \frac{k-1}{k+2} (x_k - x_{k-1}) \Big\|^2\) and  $s$ is a constant stepsize. It is worth noting that \cref{eq:AccG}, as a generalization of Nesterov's method in multiobjective optimization for the convex case, and it differs from the accelerated proximal gradient method established by Tanabe et al.  \cite{tanabe2023accelerated} in terms of distinct discretization schemes of \cref{eq:MAVD} \cite{sonntag2024fastNestrovAlgorithm}. In fact, research on multiobjective accelerated gradient methods remains insufficient, while discretization approaches can derive diverse algorithms. 
 
 {  The main challenges lie in the theoretical analysis, as detailed below:
     \begin{itemize}
     \item[$\rm (i)$] The right-hand side function of multiobjective gradient flows is often not Lipschitz continuous. Therefore, the existence proof of its trajectory solutions cannot directly apply the Cauchy-Lipschitz theorem, which increases the analytical complexity.
     
     \item[$\rm(ii)$] In Lyapunov analysis, single-objective gradient flows can utilize the property that the objective function value is no less than the minimum value to deduce the monotonicity of the Lyapunov function. However, this cannot be achieved for multiobjective gradient flows.
     \end{itemize}
     To address $\rm (i)$, Sonntag and Peitz \cite{sonntag2024fastgradientflow} employed the existence theorem for differential inclusion solutions. By constructing a differential inclusion problem equivalent to the original Cauchy problem, they completed the existence proof for trajectory solutions. To address $\rm (ii)$, constructing a Lyapunov function that does not rely on the minimum value property is key to studying multiobjective gradient flows. This requires generalizing specific Lyapunov functions from single-objective optimization, with primary references being \cite{attouch2018fast,attouch2019fast}.} 
 \subsection{Fast inertial seaech direction correction algorithm}
 { In \cite{wang2021search}, Wang et al. proposed a novel first-order algorithm called the \textit{Fast Inertial Search Direction Correction algorithm} \cref{eq:FISC} by discretizing a second-order ODE. The search direction at the current iterate is defined as a linear combination of the current gradient, the normalized gradient direction, and the past search direction. The specific iterative scheme is given by:}
 \begin{equation}\label{eq:FISC}
 \left\{\begin{aligned}
 d_{k+1}&=\frac{\ell _k-1}{\ell _{k}+\alpha -1}d_k-\frac{\alpha -3}{\ell_k +\alpha -1}\frac{\|d_k\|}{\|\nabla f(x_k)\|}\nabla f(x_k)-\nabla f(x_k),\\
 x_{k+1}&=x_k+s_kd_k.
 \end{aligned}\right. \tag{FISC}
 \end{equation}
 where $\{\ell_k\}$ is a sequence greater than  zero, which can be chosen as $\ell _k=k$. The theoretical interpretation of (\ref{eq:FISC}) remains an open problem, but the authors provided an algorithm (\ref{eq:FISCnes}) with Nesterov's acceleration characteristics. When $x_k$ satisfies the given restart condition for any $k$, it has the following iterative scheme:
 \begin{equation}\label{eq:FISCnes}
 \left\{\begin{aligned}
 y_k &=x_k+\frac{\ell_k -1}{\ell_k+\alpha -1}(x_{k}-x_{k-1})-\frac{\alpha -3}{\ell _k+\alpha -1}\frac{\|x_k-x_{k-1}\|}{\|\nabla f(x_k)\|}\nabla f(x_k),\\
 x_{k+1}&=y_k-s_k\nabla f(y_k).
 \end{aligned}\right.\tag{FISC-nes}
 \end{equation}
 We believe this iteration format can be visually interpreted with the \cref{fig:optimization_process}, which was not present in the original text. 

\begin{figure}[H]
	\centering
 \begin{tikzpicture}[
 node distance=1.5cm and 2cm,
 >=Stealth,
 font=\small
 ]
 \node (xk2) {$x_{k-1}$};
 \node[right=1.5 cm of xk2] (xk1) {$x_{k}$};
 \node[below right=0.5cm and 0.3cm of xk1] (ybar) {$x_{k}-\frac{\alpha -3}{\ell_k+\alpha -1}\frac{\|x_k-x_{k-1}\|}{\|\nabla f(x_k)\|}\nabla f(x_k)$};
 \node[right=of ybar] (yk) {$y_{k}$};
 \node[above right= 3cm and 0.8cm of yk] (xk) {$x_{k+1}$};
 
 \draw[->] (xk2) -- (xk1);
 \draw[->, dashed] (xk1) -- (ybar);
 \draw[->, dashed] (ybar) -- (yk);
 \draw[->, dashed] (yk) -- (xk);
 \draw[->] (xk1) -- (xk);
 \end{tikzpicture}
 	\caption{FISC-nes iteration diagram}
 \label{fig:optimization_process}
\end{figure}

    Compared to the FISTA algorithm, \cref{eq:FISCnes} introduces additional steepest descent information in the selection of the auxiliary variable $y_k$. Similar to traditional acceleration algorithms, both \cref{eq:FISC} and \cref{eq:FISCnes} can be viewed as discretized forms of the following equation, which we call the \textit{Accelerated Gradient Flow with Asymptotic Vanishing Normalized Gradients}: 
\begin{equation}\label{eq:AVNG}
\left\{\begin{aligned}
&r(t,x(t),\dot x(t))=\frac{\alpha -3}{t}\frac{\|\dot x(t)\|}{\|\nabla f(x(t))\|}\nabla f(x(t)),\\
&\ddot x(t)+ \frac{\alpha }{t}\dot x(t)+r(t,x(t),\dot x(t))+\nabla f(x(t))=0.
\end{aligned}\right.\tag{AVNG}
\end{equation}
It should be noted that this equation corresponds to a special choice of parameters $\beta(t)$ and $\gamma(t)$ in the following general equation:
\begin{equation}\label{eq:SDC-ODE}
\ddot x(t)+\beta (t)\dot x(t)+\gamma(t)\frac{\|\dot x(t)\|}{\|\nabla f(x(t))\|}\nabla f(x(t))=0. \tag{SDC-ODE}
\end{equation}
Selecting different parameters for \cref{eq:SDC-ODE} holds significant research potential for deriving new algorithms, though this is not the focus of this paper.
\subsection{Motivation}
{  For FISTA and \cref{eq:AccG}, the momentum term \(\frac{k-1}{k+2}(x_{k}-x_{k-1})\) is a key factor influencing the acceleration effect. During the early stages of iteration, it is often desirable for the momentum term to be larger, as the current position is still relatively far from the optimal solution. In the later stages of iteration, it is preferable for the momentum term to be smaller to mitigate the effects of oscillations.
    
    {\it So how can we change the momentum term to achieve faster convergence in the early stages of the algorithm’s iteration and reduce oscillations in the later stages?}
    
    In fact, the iterative scheme \cref{eq:FISCnes} designed by Wang et al. \cite{wang2021search} can roughly be characterized as follows: when \(x_k - x_{k-1}\) is closer to a descent direction, the iteration takes a larger step forward; when \(x_k - x_{k-1}\) is farther from a descent direction, the iteration takes a smaller step forward.
    
    Inspired by  \cref{eq:FISCnes}, {\it we aim to develop a new class of accelerated gradient flows based on \cref{eq:MAVD} and \cref{eq:AVNG}, and through discretization, obtain a new and efficient first-order accelerated method.} The method is designed for multiobjective optimization problems and is expected to exhibit strong advantages in terms of fast convergence and reduced oscillation.}

\subsection{Contributions}
{ The main contributions of this paper are as follows:}
\begin{itemize}
    \item[$\bullet$]   Firstly, we present the \textit{Multiobjective Accelerated Gradient-like Flow with Asymptotic Vanishing Normalized  Gradient}:
    \begin{equation}\label{eq:MAVNG}
    \left\{\begin{aligned}
    &r(t,x(t),\dot x(t)) =  \frac{\alpha -\beta}{t^p}\frac{\|\dot x(t)\|}{\|\proj_{C(x(t))}(0)\|}\proj_{C(x(t))}(0),\\
    &\frac{\alpha }{t}\dot x(t)+\proj_{C(x(t))+r(t,x(t),\dot x(t))+\ddot x(t)}(0)=0.
    \end{aligned}\right. \tag{MAVNG}
    \end{equation} 
    where \(p\ge 1\) and $\alpha \ge \beta \ge 3$. Note that \cref{eq:MAVNG} is a multiobjective extension obtained by uniting \cref{eq:AVNG} with \cref{eq:MAVD}; we adopt this model for the following reasons:
    { \begin{itemize}
    \item[$\rm (i)$] When \cref{eq:AVNG} is discretized into \cref{eq:FISCnes}, its normalized gradient term corresponds to the steepest-descent direction $-\nabla f(x_k)$ of the auxiliary variable at the current iterate $x_k$. Consequently, for the first part of the above expression we employ the normalized multiobjective steepest-descent direction $-\proj_{C(x_k)}(0)/\|\proj_{C(x_k)}(0)\|$.
    
    \item[$\rm(ii)$] The second part is designed by incorporating \cref{eq:MAVD}; this treatment makes a convergence analysis feasible.
    \end{itemize}}
    \item[$\bullet$] Moreover, we analysis the dynamical system \cref{eq:MAVNG} in $\mathbb{R}^n$, proving the existence of its trajectory solutions by constructing an equivalent differential inclusion. 
    	\item[$\bullet$] Using Lyapunov analysis, we derive convergence results for \cref{eq:MAVNG} under different parameter selections, as shown in the table below:
    \begin{table}[H]
        \caption{Convergence properties under different parameter selections for \cref{eq:MAVNG}. ''$\checkmark$'' indicates that the trajectory solution converges to a weakly Pareto optimal solution. The convergence rate is characterized by  the merit function proposed by Tanabe et al. \cite{tanabe2024new}. }
        \label{tab:convergence}
        \begin{center}
            \begin{tabular}{c|c|c|c}
                \hline
                & \textbf{Parameter } & \textbf{Convergence Rate} & \textbf{Convergence} \\ \hline
                \cref{thm:propositionofmeritfunction}& $\alpha \ge \beta \ge 3$, $p > 1$ & ${O}(1/t^2)$ & -- \\  
                \cref{thm:theorem4.2} & $\alpha \ge  \beta \geq 3$, $p = 1$ & ${O}(\ln^2 t / t^2)$ & -- \\
                \cref{thm:propositionofmeritfunction,thm:WeakParetoPoint} & $\alpha > \beta > 3$, $p > 1$ & ${O}(1/t^2)$ & $\checkmark$\\
                \hline 
            \end{tabular}
        \end{center}
    \end{table}
	\item[$\bullet$] 
Furthermore, we propose an algorithm similar to the discretization scheme of \cref{eq:MAVNG}, {\it Multiobjective Fast Inertial
    Search Direction Correction}{\it Method} \cref{eq:MFISC} with the following iterative scheme:
\begin{equation}\label{eq:MFISC}
\left\{\begin{aligned}
\pi_k &=\frac{k-1}{k+\alpha -1}(x_k-x_{k-1})-\frac{\alpha -3}{k+\alpha -1}\frac{\|x_k-x_{k-1}\|}{\|\proj_{C(x_k)}(0)\|}\proj _{C(x_k)}(0),\\
y_k&=x_k+\pi_k,\\
x_{k+1}&=y_k-s\proj_{C(y_k)}(\pi_k).
\end{aligned}\right. \tag{MFISC}
\end{equation}
By constructing a discrete Lyapunov function, we establish its convergence rate of $O(\ln^2k/k^2)$ under the merit function characterization (\cref{thm:sequenceConvergence}). Furthermore, all limit points of the generated iteration are weakly Pareto optimal solutions of \cref{eq:MOP}; this result is given in \cref{thm:Convergenceosequence}.
    \item[$\bullet$]  Numerical results demonstrate that, in terms of the merit function characterization, the objective function values decrease faster along the trajectories of \cref{eq:MAVNG} than along those of \cref{eq:MAVD}. The final convergence point generated by \cref{eq:MFISC}  shows no significant difference compared to \cref{eq:AccG}, while converging faster.
\end{itemize}
\subsection{Organization}
The structure of this paper is organized as follows: In \cref{sec:pre}, we present necessary preliminary knowledge; in \cref{sec:Exist}, we establish the existence of trajectory solutions for \cref{eq:MAVNG}; in \cref{sec:AsymptoticofMACG}, we discuss the asymptotic analysis of the trajectory solutions with different parameter choices and prove their convergence to weakly Pareto optimal solutions under certain conditions; in \cref{sec:Algorithm},  we provide the convergence rate of \cref{eq:MFISC}; in \cref{sec:Numercial}, we conduct numerical experiments to validate our theoretical results.
\section{Preliminary}
\label{sec:pre}
\subsection{Notation}
In this paper, \(\mathbb{R}^d\) denotes a \(d\)-dimensional Euclidean space with the inner product \(\langle \cdot, \cdot \rangle\) and the induced norm \(\|\cdot\|\). For vectors \(a, b \in \mathbb{R}^d\), we say \(a \leq b\) if \(a_i \leq b_i\) for all \(i = 1, \ldots, d\). Similarly, the relations \(a < b\), \(a\ge b\) and \(a>b\) can be defined in the same way.	 \(\R_{+}^d:=\{x\in \mathbb R^n :x\ge 0\}\). 
The set \(\Delta^d := \{ \theta \in \mathbb{R}^d :\theta \geq 0 \text{ and } \sum_{i=1}^d \theta_i = 1 \}\) is the positive unit simplex. For a set of vectors \(\{\eta_1, \ldots, \eta_d\} \subseteq \mathbb{R}^d\), their convex hull is defined as \(\conv{\{\eta_1, \ldots, \eta_d\}} := \{ \sum_{i=1}^d \theta_i \eta_i :\theta \in \Delta^m \}\). For a closed convex set \(C \subseteq \mathbb{R}^d\), the projection of a vector \(x\) onto \(C\) is  \(\proj_C(x) := \argmin_{y \in C} \|y - x\|^2\). 		

		\subsection{Pareto optimality and necessary condition}

For \cref{eq:MOP}, we have Pareto optimality which is defined as follows \cite{miettinen1999nonlinear}.

\begin{definition}\label{def:defofpareto}
	Consider the multiobjective optimization problem \cref{eq:MOP}.
	\begin{itemize}
		\item[$\rm (i)$] A point \( x^* \in \mathbb{R}^n \) is called a Pareto point or a Pareto optimal solution  if there has no \(y\in \mathbb R^n\) that \(F(y)\leq F(x^*)\) and $F(y)\neq F(x^*)$. The set of all Pareto points is called the Pareto set and is denoted by \( \mathcal{P} \).  The image \(F(\mathcal P)\) of the Pareto set $\mathcal P$ is the Pareto front.
		
		\item[$\rm (ii)$] A point \( x^* \in \mathbb{R}^n \) is called a weak Pareto point or weakly Pareto optimal solution if there has no \(y\in \mathbb R^n\) that \(F(y)<F(x^*)\). The set of all weakly Pareto optimal solutions is called the weak Pareto set and is denoted by \( \mathcal{P}_w \). The image \(F(\mathcal P_w)\) of the Pareto set $\mathcal P_w$ is the weak Pareto front.
	\end{itemize}  
\end{definition}

According to \cref{def:defofpareto}, it is clear that \( \mathcal{P} \subseteq \mathcal{P}_w \). 
Moreover, there has the following optimality conditions in \cref{eq:MOP}. 

\begin{definition}\label{def:KKTcondition}
	A point \( x^* \in \mathbb{R}^n \) is said to satisfy the Karush-Kuhn-Tucker (KKT) conditions if there exists \( \theta \in \Delta^m  \) such that
	\begin{equation}\label{eq:KKTpoint}
	\sum_{i=1}^m \theta_i \nabla f_i(x^*) = 0 .
	\end{equation}
	If \( x^* \) satisfies the KKT conditions, it is called a Pareto critical point. The set of all Pareto critical points is called the Pareto critical set and is denoted by \( \mathcal{P}_c \).
\end{definition}
\begin{remark}\label{rem:KKT=Optimal} 		
In fact, condition \cref{eq:KKTpoint} is equivalent to 
\(\proj_{C(x^*)}(0)=0\). 
The KKT conditions are necessary for weak Pareto optimality \cite[Proposition 2.1]{attouch2015multiibjective}. In the convex setting, the KKT conditions are also sufficient for weak Pareto optimality \cite[ Proposition 2.2]{attouch2015multiibjective}, and in this case, \( \mathcal{P} \subseteq \mathcal{P}_w = \mathcal{P}_c \).
\end{remark}
\subsection{Merit function}

		A merit function refers to a nonnegative function in \cref{eq:MOP} that attains zero only at weakly Pareto optimal solutions. Recent studies employed the merit function established by Tanabe et al. \cite{tanabe2024new} to assess the convergence rate of objective values \cite{bot2024inertial,luo2025accelerated,sonntag2024fastgradientflow,sonntag2024fastNestrovAlgorithm,tanabe2023accelerated,tanabe2023convergence,yang2024global}. In this paper, we also consider the merit function
\begin{equation}\label{eq:meritfunction} 
\varphi(x) := \sup_{z \in \mathbb{R}^n} \min_{i=1,\ldots,m} f_i(x) - f_i(z) .
\end{equation}
The following theorem forms the basis for our ability to use merit functions for convergence analysis.
\begin{thm}\label{thm:weakpareto} 
	Let \( \varphi:\mathbb R^n \to \R \) be defined as in \cref{eq:meritfunction} with lower semicontinous functions $f_i$ for all $i=1,\cdots,m$. Then for \cref{eq:MOP}, we have
    \begin{itemize}
		\item[\(\rm (i)\)] \( \varphi(x) \geq 0 \) for all \( x \in \mathbb{R}^n \);
		
		\item[\(\rm (ii)\)]  \( x \in \mathbb{R}^n \) is a weakly Pareto optimal solution
        if and only if \( \varphi(x) = 0 \);
		
		\item[\(\rm (iii)\)]  \( \varphi(x) \) is lower semicontinuous.
    \end{itemize}
\end{thm}
\begin{proof}
	See \cite[Theorem 3.1, 3.2]{tanabe2024new}.\qed 
\end{proof} 

Given the definition of \( \varphi(x) \), it represents the global maximum of \(h(z)= \min_{i=1,\ldots,m} f_i(x) - f_i(z) \) over \( \mathbb{R}^n \), making its computation challenging. Fortunately, we will later find that the merit function can be discussed within a smaller set. To this end, we introduce the concept of level sets.

\begin{definition}\label{def:levelset}
	Let \( F: \mathbb{R}^n \to \mathbb{R}^m \), \( F(x) = (f_1(x), \ldots, f_m(x))^\top \), be a vector-valued function, and let \( a \in \mathbb{R}^m \). The  level set is defined as $\mathcal{L}(F, a) := \{ x \in \mathbb{R}^n : F(x) \leq a \}$. 
	Moreover, we denote
	$\mathcal{LP}_w(F, a) := \mathcal{L}(F, a) \cap \mathcal{P}_w.$ 
\end{definition}

\subsection{Assumption} 
Similar to the setup in \cite{bot2024inertial,burachik2017new,luo2025accelerated}, the following assumptions hold throughout the paper:

\begin{assum} \label{assum:Lj-muj} Each function $f_i:\mathbb{R}^n\to \mathbb{R}$, $i=1,\cdots,m$ is below bounded, convex and continuously differentiable, with Lipschitz gradients, i.e., { there exist $L_i\in \R$ such that } $\|\nabla f_i({x})-\nabla f_i({y})\|\le L_i\|{x}-{y}\|$
for all  ${x},{y}\in \mathbb{R}^n$. Let $L:\max_{i=1,\cdots,m} L_i$. 
\end{assum}

\begin{assum}\label{assum:levelset}	 There exists $1\le j_0\le m$ such that the level set $\mathcal{L}(f_{j_0},\alpha)=\{{x}\in\mathbb{R}^n :f_{j_0}({x})\le \alpha\}$ is bounded for any $\alpha\in \mathbb{R}$. 
\end{assum}

\begin{assum}  \label{assum:alpha-pw} Let $\alpha\in \mathbb{R}$ such that the level set $\mathcal{L}(F,\alpha)$ is nonempty. For any ${x}\in \mathcal{L}(F,\alpha)$, the set $\mathcal{L}\mathcal{P}_w(F,F({x})):=\mathcal{P}_w\cap \mathcal{L}(F,F({x}))$ is nonempty. 
\end{assum}

\begin{remark}\label{rem:levelsetisbounded0}  By \cref{assum:levelset}, note that for any $\alpha\in \mathbb{R}$, $\mathcal{L}(F,\alpha)$ is bounded with radius $R_{j_0}(\alpha)$. Moreover, for any ${x}\in \mathbb{R}^n$, the level set of the function $h({z}):=\max _{i=1,\cdots,m}\left[f_i({z})-f_i({x})\right]$ is nonempty and bounded. Based on the smoothness of the objective function in \cref{assum:Lj-muj}, we can conclude that $\mathcal P_w$ is a closed set \cite[Chapter 6, Theorem 1.1]{luc1989vector}.
\end{remark}

Regarding the merit function $\varphi(x)$, we can consider the set-valued mapping
\begin{equation}\label{eq:setvaluedmapofmerit}
S: \mathbb R^n \rightrightarrows \mathbb R^m ,\qquad x \mapsto \underset{z \in \mathbb R^n}{\argmin} \max_{i=1,\ldots,m} f_i(z) - f_i(x).
\end{equation}
If there exists $z^* \in S(x) \neq \varnothing$, then the merit function $\varphi(x) = \min_{i=1,\ldots,m} f_i(x) - f_i(z^*)$. The following theorem states an important property of \( S(x(t)) \) which is similar to \cite[Proposition 2.4]{bot2024inertial}.

\begin{thm}\label{thm:boundedz}
	Let $a \in \mathbb R^n_+$ as described in \cref{assum:levelset}, and suppose the  function $x: [t_0, +\infty) \to \mathbb R^n$ satisfies $x(t) \in \mathcal L(F, F(x(t_0)) + a)$ for all $t \ge t_0$. The set-valued mapping $S$ is defined as \cref{eq:setvaluedmapofmerit}. Then,
		\begin{itemize} 
		\item[\(\rm (i)\)]  $S(x(t)) \subseteq \mathcal L\mathcal P_w(F, F(x(t_0)) + a)$,
		
		\item[\(\rm (ii)\)]  The function
		\begin{equation}\label{eq:defofzt}
		z^*(t) := \underset{z \in S(x(t))}{\argmin} \| z \|^2 ,
		\end{equation}
		is bounded for all  $t\ge t_0$.
        \end{itemize} 
\end{thm}
\begin{proof}
	By the \cref{rem:levelsetisbounded0}, $h(z) = \max_{i=1,\ldots,m} f_i(z) - f_i(x(t))$ is a continuous convex function with nonempty and bounded level sets for all $t \ge t_0$. Fix some $t\ge t_0$,  $S(x(t))$ is  nonempty and compact  by the Weierstrass Theorem \cite[Proposition 3.2.1]{bertsekas2009convex}. Furthermore, since $\| z \|^2$ is strongly convex, there exists a unique $z^*(t)$ such that \cref{eq:defofzt} holds.
	
	We now prove $\rm (i)$. 
	Let $z \in S(x(t))$. Then,
	\begin{equation*}
	\max_{i=1,\ldots,m} f_i(z) - f_i(x(t)) \le \max_{i=1,\ldots,m} f_i(x(t)) - f_i(x(t)) = 0.
	\end{equation*} 
	Moreover, since $x(t) \in \mathcal L(F, F(x_0) + a)$, we get
	\begin{equation*}
	f_i(z) \le f_i(x(t)) \le f_i(x(t_0)) + a_i, 
	\end{equation*}
	for all $i = 1, \ldots, m$, and therefore $z \in \mathcal L(F, F(x(t_0)) + a)$. To prove  (i) by contradiction, suppose $z \notin \mathcal L\mathcal P_w(F, F(x(t_0)) + a)$ and hence there exists $y \in \mathbb R^n$ such that
	$f_i(y) < f_i(z)$
 	for all  $i = 1, \cdots, m$. Hence,
	\begin{equation*}
	\max_{i=1,\ldots,m} f_i(y) - f_i(x(t)) < \max_{i=1,\ldots,m} f_i(z) - f_i(x(t)),
	\end{equation*}
	which contradicts $z \in S(x(t))$. The boundedness of $z^*(t)$ can be derived from the boundedness of the level set $\mathcal{L}(F, F(x_0) + a)$, which proves $\rm{(ii)}$.  \qed
\end{proof}
\begin{coro}\label{coro:sequenceofxk}
	Let \( a \in \mathbb{R}_+^m \) as described in \cref{assum:levelset}, and suppose the sequence \( \{x_k\} \subseteq \mathbb{R}^n \) satisfy \( x_k \in \mathcal{L}(F, F(x_0) + a) \) for all \( k \geq 1 \) and $x_1=x(t_0)$. Then,
		\begin{itemize} 
		\item[\(\rm (i)\)] \( S(F(x_k)) \subseteq \mathcal{L}\mathcal{P}_w(F, F(x_1) + a) \),
		
		\item[\(\rm (ii)\)] The sequence
		\begin{equation*}
		z_k^* := \underset{z \in S(x_k)}{\argmin} \|z\|^2,
		\end{equation*}
		is bounded for all $k\ge 1$, where \( S(x_k) := \argmin_{z \in \mathbb{R}^n} \max_{i=1,\dots,m} f_i(z) - f_i(x_k) \).
        \end{itemize} 
\end{coro}
\begin{proof}
	By \cref{thm:boundedz}, the conclusion follows immediately. \qed 
\end{proof}	

\section{Existence of solutions}\label{sec:Exist}
In this section, we  prove the existence of solutions to the differential equation, using a method analogous to that described in \cite{bot2024inertial,sonntag2024fastgradientflow,sonntag2024fastNestrovAlgorithm}. Consider the following Cauchy problem \cref{eq:Cauchy-Problem}:
\begin{equation}\label{eq:Cauchy-Problem}
\left.
\begin{aligned}
&\frac{\alpha}{t} \dot{x}(t) + \mathbf{proj}_{C(x(t)) + \ddot{x}(t)+r(t,x,\dot x)}(0) = 0 ,
&x(t_0) = x_0, \dot{x}(t_0) = v_0.
\end{aligned}
\right.\tag{CP}
\end{equation}
where $r(t,x,\dot x) = \frac{\alpha - \beta}{t^p} \frac{\|\dot{x}(t)\|}{\|\mathbf{proj}_{C(x(t))}(0)\|} \mathbf{proj}_{C(x(t))}(0)$ and $t_0\ge 1$. To prove the existence of solutions to \cref{eq:Cauchy-Problem}, analogous to the discussion in \cite{sonntag2024fastgradientflow}, it suffices to focus on the existence of solutions to the following differential inclusion \cref{eq:DI}:
\begin{equation}
\left.
\begin{aligned}
&(\dot{u}(t), \dot{v}(t)) \in G(t, u(t), v(t)) ,
&(u(t_0) ,  v(t_0) )= (u_0, v_0),\\
\end{aligned}
\right.\tag{DI}
\label{eq:DI}
\end{equation}
where
\begin{equation}
\begin{aligned}
&G: [t_0, +\infty) \times \mathbb{R}^n \setminus \mathcal{P}_w \times \mathbb{R}^n \rightrightarrows \mathbb{R}^n \times \mathbb{R}^n,\\
& (t, u, v) \mapsto \left\{ \left( v, -\frac{\alpha}{t}v - \underset{g \in C(u) + r(t,u,v)}{\argmin} \, \langle g, -v \rangle \right) \right\}.
\end{aligned}
\label{eq:G_def}
\end{equation}

\subsection{Existence of solutions to the differential inclusion \cref{eq:DI}} 

For $G$, we have the following relevant properties:
\begin{lem}
	Assume  \cref{assum:Lj-muj,assum:levelset,assum:alpha-pw} holds true. The set-valued mapping $G$ defined in \eqref{eq:G_def} has the following properties:
	\begin{itemize}
		\item[$\rm (i)$] For any $(t, u, v) \in [t_0, +\infty) \times \mathbb{R}^n \setminus \mathcal{P}_w \times \mathbb{R}^n$, the set $G(t, u, v) \subset \mathbb{R}^n \times \mathbb{R}^n$ is convex and compact.
		\item[$\rm (ii)$] $G$ is upper semicontinuous.
		\item[$\rm (iii)$] The following mapping is locally compact:
		\begin{equation*}
		\phi: [t_0, +\infty) \times \mathbb{R}^n\setminus\mathcal P_w \times \mathbb{R}^n \to \mathbb{R}^n \times \mathbb{R}^n, \quad (t, u, v) \mapsto \mathbf{proj}_{G(t,u,v)}(0).
		\end{equation*}
		
	\end{itemize}
	\label{lem:SSS}
\end{lem}
\begin{proof}
(i) For any $(t,u,v)$, note that $C(u) + r(t,u,v)$ is convex and compact because it is the sum of the convex hull of a finite set and a fixed vector. Therefore, the set $\arg\min_{g \in C(u) + r(t,u,v)} \langle g, -v \rangle$ is convex and compact. On the other hand, since convexity is preserved under Cartesian products, $G(t,u,v)$ is convex and compact.  

(ii) Let $\Omega = [t_0, +\infty) \times \R^n \setminus \mathcal{P}_w \times \R^n$. By \cite[Proposition 3.8]{Attouch2014}, $\proj_{C(u)}(0)$ is continuous in $u$. Thus, $r(t,u,v) = \frac{\alpha - \beta}{t} \frac{\|v\|}{\|\proj_{C(u)}(0)\|} \proj_{C(u)}(0)$ is continuous on $\{(t,u,v) \in \Omega :\proj_{C(u)}(0) \neq 0\}$. By the convexity of $f_i$ for all $i=1,\cdots,m$ and \cref{rem:KKT=Optimal}, $r(\cdot)$ is continuous on $\Omega$. Since $C(\cdot)$ is continuous on $\R^n \backslash \mathcal{P}_w$, the mapping $C(\cdot) + r(\cdot)$ is continuous on $\Omega $. Furthermore, by \cite[Theorem 3B.5]{dontchev2009implicit}, the set-valued mapping $\arg\min_{g \in C(u) + r(t,u,v)} \langle g, -v \rangle$ is upper semicontinuous. Similar to the proof of \cite[Proposition 3.1]{sonntag2024fastgradientflow}, we conclude that $G$ is upper semicontinuous.  

(iii) This follows directly from (ii).	\qed
\end{proof}
Below we demonstrate the existence of a  trajectory solution for th\cref{eq:DI}.

\begin{thm}
	Assume  \cref{assum:Lj-muj,assum:levelset,assum:alpha-pw} holds true, then for any initial condition $(u_0, v_0) \in \mathbb{R}^n \setminus \mathcal{P}_w \times \mathbb{R}^n$, there exist $T > t_0$ and an absolutely continuous function $(u(\cdot), v(\cdot))$ defined on $[t_0, T]$ that is a solution to the differential inclusion \eqref{eq:DI}.
\end{thm}
\begin{proof}
	Follows directly from  \cref{lem:SSS} and \cref{lem:exist}\qed 
\end{proof}

\begin{thm}\label{eq:exsitenceofDI}
	Assume  \cref{assum:Lj-muj,assum:levelset,assum:alpha-pw} holds true, then for any initial condition $(u_0, v_0) \in \mathbb{R}^n \setminus \mathcal{P}_w \times \mathbb{R}^n$, there exists a function $(u(\cdot), v(\cdot))$ defined on $[t_0, T)\subseteq [1,+\infty )$ that is absolutely continuous on any closed subinterval and is a solution to the differential inclusion \eqref{eq:DI}, satisfying:
	\begin{itemize}
		\item[{$\rm (i)$}] $(u, v)$ cannot be extended beyond $[t_0, T)$.
		\item[{$\rm (ii)$}] Either $T = +\infty$, or $T < +\infty$ and $u_T := \lim_{t \to T^{-}} u(t) \in \mathcal{P}_w$.
	\end{itemize}
\end{thm}	
\begin{proof}
	Completely analogous to that described in \cite[Theorem 3.5]{sonntag2024fastgradientflow}. Define
	\begin{equation*}
    \begin{aligned} 
	\mathfrak{S} := \{ (u,v) :&\  [t_0, T) \to \mathbb{R}^n \times \mathbb{R}^n \\&: T \in [t_0, +\infty], (u,v) \text{ is a solution of \eqref{eq:DI} on } [t_0, T) \}
    \end{aligned} 
	\end{equation*}
	For the two solutions in \eqref{eq:DI}, $(u_1(\cdot),v_1(\cdot)):[t_0,T_1)\to \R^n\times \R^n$ and $(u_2(\cdot),v_2(\cdot)):[t_0,T_2)\to \R^n\times \R^n$, we define the partial ordering $\preccurlyeq$ in $\mathfrak S$ as follows:
	\begin{equation*} 
	\begin{aligned} 
	&(u_1(\cdot ),v_1(\cdot))\preccurlyeq (u_2(\cdot),v_2(\cdot)) :\\ &\Leftrightarrow
	T_1\le T_2\text{ and }(u_1(t),v_1(t))=(u_2(t),v_2(t))\text{ for all }t\in[t_0,T_1).
	\end{aligned}
	\end{equation*}
	Then, there exists a maximal element $(u(\cdot), v(\cdot))$ defined on $[t_0, T)$ by Zorn's Lemma. If $T = +\infty$, we are done. If $T < +\infty$, for $h(t) := \|(u(t), v(t)) - (u(t_0), v(t_0))\|$
    we have
	\begin{equation*}
	\frac{d}{dt}\frac12 h^2(t) \le \|(\dot{u}(t), \dot{v}(t))\| h(t).
	\end{equation*}
	Note that $(\dot u(t),\dot v(t)) = (v,-\frac{\alpha}{t}v-g)$, $g\in\argmin_{C(u)+r(t,u,v)}\langle g,-v\rangle $, 
	\begin{align*}
	\|(\dot{u}(t), \dot{v}(t))\| &\le \|v\| + \left\| -\frac{\alpha}{t}v - g \right\| \\
	&\le \left(1 + \frac{\alpha}{t}\right) \|v\| + \max_{\theta \in \Delta^m} \left\| \sum_{i=1}^m \theta_i \nabla f_i(u) \right\| + \|r(t,u,v)\| \\
	&\le \left(1 + \frac{2\alpha - \beta}{t_0}\right) \|v\| + L \|u\| + \max_{i=1,\dots,m} \|\nabla f_i(0)\| \\&\le c(1 + \|(u,v)\|).
	\end{align*}
	with $c= \sqrt 2 \max \{(1+\frac{2\alpha -\beta}{t_0}),L,\max_{i=1,\cdots,m}\|\nabla f_i(0)\|\}$. Define $\bar{c} = c(1 + \|(u(t_0), v(t_0))\|)$. Applying the triangle inequality, we get
	\begin{equation*}
	\|(\dot{u}(t), \dot{v}(t))\| \le \bar{c}(1 + \|(u(t), v(t)) - (u(t_0), v(t_0))\|).
	\end{equation*}
	As argued in \cite{sonntag2024fastgradientflow}, for any $\varepsilon > 0$ and almost all $t \in [t_0, T - \varepsilon]$, we have
	$h(t) \le \bar{c} T \exp(\bar{c} T)$.
	By the independence of $\varepsilon$ and $t$, we have $h$ is uniformly bounded on $[t_0, T)$. This implies that $\dot{u}(t)$, $\dot{v}(t)$, $u(t)$, and $v(t)$ are bounded. Since $u$ and $v$ are absolutely continuous on $[t_0, t]$ for any $t \in [t_0, T)$,
	\begin{equation*}
	u(t) = u_0 + \int_{t_0}^t \dot{u}(s) \, ds, \quad v(t) = v_0 + \int_{t_0}^t \dot{v}(s) \, ds,
	\end{equation*}
	the left limits of $u(t)$ and $v(t)$ at $T$ exist. Thus, define
	\begin{equation*}
	u_T := u_0 + \int_{t_0}^T \dot{u}(s) \, ds \in \mathbb{R}^n, \quad v_T := v_0 + \int_{t_0}^T \dot{v}(s) \, ds \in \mathbb{R}^n.
	\end{equation*}
	Similar to \cite{sonntag2024fastgradientflow}, If $u_T \notin P_w$, then there exists a solution $(u^*(\cdot),v^*(\cdot))$ such that $(u(\cdot),v(\cdot)) \ne (u^*(\cdot),v^*(\cdot))$ and $(u(\cdot),v(\cdot)) \preccurlyeq (u^*(\cdot),v^*(\cdot))$, contradicting the maximality of $(u(t), v(t)) $. \qed
\end{proof}

\subsection{Existence of solutions to the differential equation \cref{eq:Cauchy-Problem}}
We establish the existence of solutions to \cref{eq:Cauchy-Problem} in a manner similar to \cite{sonntag2024fastNestrovAlgorithm}.
\begin{definition}\label{def:local-solution}
    We call a function $x:[t_0,T)$, $t\mapsto x(t)$ with $T\in(t_0,+\infty]$ a local solution to $\cref{eq:Cauchy-Problem}$ if it satisfies the following conditions:
    	\begin{itemize}
        \item[$\rm (i)$] $x(t) \in C^1([t_0, T))$, i.e., $x(t)$ is continuously differentiable on $[t_0, T)$;
        \item[$\rm (ii)$] For any $t_0 \le T' < T$, $\dot{x}(t)$ is absolutely continuous on $[t_0, T']$;
        \item[$\rm (iii)$] There exists a measurable function $\ddot{x}: [t_0, T) \to \mathbb{R}^n$ such that $\dot{x}(t) = \dot{x}(t_0) + \int_{t_0}^t \ddot{x}(s) \, ds$ for $t \ge t_0$;
        \item[$\rm (iv)$] $\dot{x}$ is differentiable almost everywhere and $\frac{d}{dt} \dot{x}(s) = \ddot{x}(s)$ holds for almost all $t \in [t_0, T)$;
        \item[$\rm (v)$] $\frac{\alpha}{t} \dot{x}(t) + \mathbf{proj}_{C(x(t)) + \ddot{x}(t)}(0) = 0$ holds for almost all $t \in [t_0, T)$;
        \item[$\rm (vi)$] $x(t_0) = x_0$ and $\dot{x}(t_0) = v_0$ hold.
    \end{itemize}
\end{definition}
{ \begin{definition}\label{def:global-solution}
    We call a function $x:[t_0,T)$, $t\mapsto x(t)$ with $T\in(t_0,+\infty]$ a global solution to $\cref{eq:Cauchy-Problem}$ if it satisfies the following conditions:
    \begin{itemize}
        \item[$\rm(i)$] $x$ is a local soution;
        \item[$\rm (ii)$] For any given $T' \ge T$, if the function $y(\cdot)$ defined on $[t_0, T')$ is a local solution to \cref{eq:Cauchy-Problem} and satisfies $y(t) = x(t)$ for all $t \in [t, T)$, then it must hold that $T' = T$.
    \end{itemize}
\end{definition}}
\begin{thm}
For all initial values $(x_0,v_0)\in \R^n\setminus\mathcal P_w\times \R^n$, there exists a function $x(\cdot)$ which is global solution of \cref{eq:Cauchy-Problem} in the sense of \cref{def:global-solution}. 
\end{thm}
\begin{proof}
	Let $(u(t), v(t))$ be a maximal solution to the differential inclusion with $T\in(t_0,+\infty]$ and the condition $\rm (ii)$  in \cref{def:global-solution} holds. Let $x(t)=u(t)$, then by the definition of the set-value mapping $G$ showed in \cref{eq:G_def}, we have $\dot x(t)=v(t)$.  This implies that $v(t)$ is  continuous on $[t_0, T)$, hence $x(t) \in C^1([t_0, T)$. This proves (i) and (ii) in the \cref{def:local-solution}. By Definition 3.3 in \cite{Attouch2014}, (iii) and (iv) hold. By \cref{eq:exsitenceofDI,lem:lemmaA4}, (v) and (vi) hold. Therefore, $x(t)$ is a local solution which is the condition $\rm (i)$ in \cref{def:global-solution}.  \qed
\end{proof}

\section{Asymptotic analysis for gradient-like flow}
\label{sec:AsymptoticofMACG}
This section discusses the properties of the global solution in $[t_0,+\infty )$ to \cref{eq:Cauchy-Problem} under the initial condition $(x(t_0),\dot{x}(t_0)) = (x_0,0)$. 
In \cref{thm:WeakParetoPoint} , we prove that when $\alpha > \beta > 3$ and $p>1$, $x(t)$ converges to a weakly Pareto optimal solution. To establish this result, we first demonstrate in  \cref{thm:propositionofmeritfunction} that for $\alpha \ge  \beta \geq 3$ and $p>1$, the convergence rate  satisfies $\varphi(x(t)) = O(1/t^2)$. Additionally, we establish a weaker result for $p=1$: $\varphi(x(t)) = O({\ln^2 t}/{t^2})$.

We recall that \cref{assum:Lj-muj,assum:levelset,assum:alpha-pw} hold throughout the paper, which simplifies the statements of our subsequent theorems. We give an important notation:

\begin{itemize}
	\item[$\bullet$] The level set $\mathcal{L}(F,F(x_0))$ is bounded, and we denote its radius by $R$, i.e.,
	\begin{equation}\label{eq:Levelsetbounded}
	R := \sup_{x \in \mathcal{L}(F,F(x_0))} \|x\|.
	\end{equation}
\end{itemize}
Before proving the convergence, we provide two important propositions:

\begin{lem}\label{lem:energyfunctionW}
	Let $x: [t_0, +\infty) \to \mathbb{R}^n$ be a global solution to \cref{eq:Cauchy-Problem} with $p\ge 1$ and $t_0\ge 1$. For $i = 1, \dots, m$, define the global energy
	\begin{equation}\label{eq:energyfunctionW}
	\mathcal{W}_i: [t_0, +\infty) \to \mathbb{R}, \quad t \mapsto  f_i(x(t)) + \frac{1}{2} \|\dot{x}(t)\|^2.
	\end{equation}
	Then,  we have $\frac{d}{dt} \mathcal{W}_i(t) \leq -\frac{\beta}{t} \|\dot{x}(t)\|^2$ for $i = 1, \dots, m$ and almost all $t \in [t_0, +\infty)$. Moreover, $\lim_{t \to \infty} \mathcal{W}_i(t) = \inf_{t \geq t_0} \mathcal{W}_i(t) \in \mathbb{R}$ exists.
\end{lem}
\textit{Proof}
	The function $\mathcal{W}_i$ is almost everywhere differentiable on $[t_0, +\infty)$, and
	\begin{equation}\label{eq:energyfunctiondiff}
	\begin{aligned}
	\frac{d}{dt} \mathcal{W}_i(t)& = \Big<  \nabla f_i(x(t)), \dot{x}(t) \Big>  + \Big< \ddot{x}(t), \dot{x}(t) \Big>.
	\end{aligned}
	\end{equation}
	Since $-\frac{\alpha}{t} \dot{x}(t) = \proj_{C(x(t)) + r(t,x(t),\dot x(t))}(0)$ and  $\nabla f_i(x(t)) \in C(x(t))$, we have,
	\begin{equation*}
	\left\langle \ddot{x}(t) + \frac{\alpha}{t} \dot{x}(t)  +\nabla f_i(x(t))+r(t,x(t),\dot x(t)),\dot x(t) \right\rangle \leq 0 .
	\end{equation*}
By rearranging the terms, we obtain the first result. By the \cref{assum:Lj-muj}, $f_i$ is bounded below, then we have the second result.\qed 
The following proposition shows that $f_i(x(t))$ has an upper bound $f_i(x(t_0))$. 

\begin{lem}\label{lem:boundedsolution}
	Let $x: [t_0, +\infty) \to \mathbb{R}^n$ be a global solution to \cref{eq:Cauchy-Problem} with initial condition \((x(t_0),\dot x (t_0))=(x_0,0) \). Then, for all $i = 1, \dots, m$ and all $t \in [t_0, +\infty)$, we have
	\begin{equation}\label{eq:level-bounded-x0} 
	f_i(x(t)) \leq f_i(x_0) .
	\end{equation}
	$\rm i.e.,$ $x(t) \in \mathcal{L}(F, F(x_0))$ for $t \geq t_0$. Furthermore,  $x(t)$ is a bounded solution.
\end{lem}
\begin{proof}
Based on \cref{lem:energyfunctionW}, using the non-increasing property of $\mathcal{W}_i$, it is straightforward to obtain \cref{eq:level-bounded-x0}. Furthermore, combining this with the boundedness of  $\mathcal{L}(F, F(x_0))$ given in \cref{assum:levelset}, the boundedness of $x(t)$ is also proved. \qed
\end{proof}


\begin{definition} Let \(x: [t_0, +\infty) \to \mathbb{R}^n\) be a global solution of \cref{eq:Cauchy-Problem} and \(z\in \R^n\). For \(t \ge t_0\), define the Lyapunov function:
	\begin{equation}\label{eq:Lyapunov2}
    \begin{aligned}
	\mathcal E_z(t)&=\frac{t^2}{2(\alpha -1)}\min_{i=1,\cdots,m}\Big(f_i(x(t))-f_i(z)\Big)\\&\quad +\frac{(\alpha -3)t^2}{4(\alpha -1)^2}\|\dot x(t)\|^2+\frac12\left\|x(t)-z+\frac{t}{\alpha -1}\dot x(t)\right\|^2,
	\end{aligned}
    \end{equation}

\end{definition}

\begin{lem}\label{lem:LyapunovInequal}
	Let \(x: [t_0, +\infty) \to \mathbb{R}^n\) be a global solution of \cref{eq:Cauchy-Problem} and \(z\in \R^n\). For $p\geq 1$, we have $\mathcal E(\cdot)$ is differentiable for almost all $t \ge t_0$, and
	\begin{equation*}
	\frac{d}{dt}\mathcal E_z(t)\leq \frac{(3-\beta )}{2(\alpha -1)}t\|\dot x(t)\|^2+\frac{\alpha -\beta }{(\alpha -1)t^{p-1}}\|\dot x(t)\|\|x(t)-z\|.
	\end{equation*}
    for almost all  $t\ge t_0$. 
\end{lem}
{\it Proof\ } The proof of this lemma is based on simple inequality techniques, which we defer to the \cref{appendix:ProofofLyapunov}.\qed 

\begin{thm} \label{thm:propositionofmeritfunction}
	 Let \(x: [t_0, +\infty) \to \mathbb{R}^n\) be a global solution of \cref{eq:Cauchy-Problem} with $p>1$. Then, as \(t \to +\infty\),
	\begin{itemize}
		\item[\(\rm (i)\)] \(\|\dot{x}(t)\| = {O}(1/t)\) if \(\alpha \ge \beta \geq 3\);
		
		\item[\(\rm (ii)\)] \(\varphi(x(t)) = {O}(1/t^{2})\) if  \(\alpha \ge \beta \ge 3\);
		
		\item[\(\rm (iii)\)] \(t \|\dot{x}(t)\|^2 \in L^1( [t_0, +\infty),\R)\) if $\alpha >\beta >3$.
	\end{itemize}
\end{thm}
\begin{proof}
	By the \cref{lem:LyapunovInequal}, we have
	\begin{equation}\label{eq:lypunov-1}
	\begin{aligned}
	\mathcal E_z(t)&-\mathcal E_z(t_0)\\&\leq -\frac{(\beta -3)}{2(\alpha -1)}\int_{t_0}^ts\|\dot x(s)\|^2ds+\frac{\alpha -\beta }{\alpha -1}\int_{t_0}^t\frac{\|\dot x(s)\|\|x(s)-z\|}{s^{p-1}}ds\\
	&\le -\frac{\beta -3}{2(\alpha -1)}\int_{t_0}^t s\|\dot x(s)\|^2ds+\frac{\alpha -\beta }{\alpha -1}\cdot \Big(R+\|z\|\Big)\int_{t_0}^t\frac{\|\dot x(s)\|}{s^{p-1}}ds,
	\end{aligned}
	\end{equation}
	with $R$ defined in \cref{eq:Levelsetbounded}. Take $z=z^*(t)$ as shown in \cref{thm:boundedz}., we get 
	\begin{equation}\label{eq:Inequal_19}
	\begin{aligned}
	\frac{t^2}{2(\alpha -1)}\varphi(x(t))&+\frac{(\alpha -3)t^2}{4(\alpha -1)^2}\|\dot x(t)\|^2+\frac{\beta -3}{2(\alpha -1)}\int_{t_0}^t s\|\dot x(s)\|^2ds\\&\le \mathcal E_{z^*(t)}(t)+\frac{\beta -3}{2(\alpha -1)}\int_{t_0}^t s\|\dot x(s)\|^2ds\\
    &\le \mathcal E_{z^*(t)}(t_0)+\frac{2R(\alpha -\beta )}{\alpha -1}\int_{t_0}^t \frac{\|\dot x(s)\|}{s^{p-1}}ds\\
	&\le \frac{t_0^2}{2(\alpha-1)}\varphi(x_0)+2R^2+\frac{2R(\alpha -\beta )}{\alpha -1}\int_{t_0}^t \frac{\|\dot x(s)\|}{s^{p-1}}ds.
	\end{aligned}
	\end{equation}
	Furthermore, since $\varphi(x(t))\ge 0$,  we have
	\begin{equation}\label{eq:thm3.1-21}
	\frac{1}{2}t^2\|\dot x(t)\|^2+\frac{\beta -3}{2(\alpha -1)}\int_{t_0}^t s\|\dot x(s)\|^2ds\le \frac12C_1^2+\int_{t_0}^t\frac{C_2}{s^{p}}s\|\dot x(s)\|ds,
	\end{equation}
	where $\frac12C_1^2=\frac{2(\alpha -1)^2}{\alpha -3}(\frac{t_0^2}{2(\alpha-1)}\varphi(x_0)+2R^2)$ and $C_2=\frac{4R(\alpha -\beta )(\alpha -1)}{\alpha -3}$. By the \cref{lem:lemmaA3}, we have
	\begin{equation}\label{eq:tilde-C1}
	t\|\dot x(t)\|\le C_1+\int_{t_0}^t\frac{C_2}{s^{p}}ds\le C_1+\int_{t_0}^{+\infty}\frac{C_2}{s^{p}}ds:=\widetilde{C_1},
	\end{equation}
	and then, we have $\|\dot x(t)\|\le \frac{\widetilde {C_1}}{t}$, $\rm i.e.$ \(\|\dot x(t)\|=O(1/{t})\). By combining \cref{eq:Inequal_19,eq:thm3.1-21,eq:tilde-C1} , we obtain
	\begin{equation*}
	\frac{t^2}{2}\varphi (x(t))\le \frac{t_0^2}{2}\varphi(x_0)+2(\alpha -1)R^2+{2R(\alpha -\beta )}\int_{t_0}^{+\infty }\frac{\widetilde C_1}{t^p}dt<+\infty ,
	\end{equation*}
	which implies that $\varphi(x(t))=O(1/t^2)$. If $\beta >3$, according to \cref{eq:thm3.1-21,eq:tilde-C1}, we have
	\begin{equation*}
	\frac{\beta-3}{2(\alpha -1)}\int_{t_0}^{+\infty }t\|\dot x(t)\|^2dt \le 4\frac{\alpha -\beta }{\alpha -1}R^2\int_{t_0}^{+\infty }\frac{\widetilde C_1}{t^{p-1}}dt<+\infty ,
	\end{equation*}
	which implies that \(t\|\dot x(t)\|^2\in L^1([t_0,+\infty),\R )\). \qed
\end{proof}
\begin{thm}\label{thm:theorem4.2}
	 Let \(x(\cdot)\) be a global solution of \cref{eq:Cauchy-Problem} with \(p = 1\), \(\alpha \ge \beta \ge 3\).  Then, as \(t \to +\infty\), we have $\varphi(x(t))=O(\ln^2 t/t^2)$. 
\end{thm}
\begin{proof}
	Similar to \cref{eq:lypunov-1}, we have
	\begin{equation*}
	\begin{aligned}
	\mathcal E_z(t)-\mathcal E_z(t_0)
	&\le \frac{\alpha -\beta }{\alpha -1}\cdot 2\Big(R+\|z\|\Big)\int_{t_0}^t\|\dot x(s)\|ds .
	\end{aligned}
	\end{equation*}
	Furthermore, Taking $z=z^*(t)$ and by the simple computation, we have
	\begin{equation}\label{eq:lyapunov-2}
	\frac{1}{2}t^2\|\dot x(t)\|^2\le\frac{\alpha -1}{\alpha -3}t^2\varphi(x(t))+\frac{1}{2}t^2\|\dot x(t)\|^2\le  \frac12C_1^2+\int_{t_0}^t\frac{C_2}{s}s\|\dot x(s)\|ds,
	\end{equation}
	where $\frac12C_1^2=\frac{2(\alpha -1)^2}{\alpha -3}(\frac{t_0^2}{2(\alpha-1)}\varphi(x_0)+2R^2)$ and $C_2=\frac{4R(\alpha -\beta )(\alpha -1)}{\alpha -3}$. By the \cref{lem:lemmaA3}, we have
	\begin{equation*}
	t\|\dot x(t)\|\le C_1+\int_{t_0}^t\frac{C_2}{s}ds\le C_1+C_2\ln{t}+C_2|\ln{t_0}|,
	\end{equation*}
	then, 
	\begin{equation*}
	\|\dot x(t)\|\le\left(  {\sup_{t\ge t_0}\left (\frac{C_1+C_2|\ln{t_0}|}{\ln{t}}\right)+C_2}\right)\frac{\ln t}{t}:=\widetilde{C_2}\frac{\ln t}{t}.
	\end{equation*}
	${\rm i.e.}\  \|\dot x(t)\|=O(\ln{t}/t)$. Moreover, combineing with \cref{eq:lyapunov-2}, we have
	\begin{equation*}
    \begin{aligned}  
	\frac{\alpha -1}{\alpha -3}t^2\varphi (x(t))&\le \frac12C_1^2+C_2\widetilde{C_2}\int_{t_0}^t\frac{\ln s}{s}ds\\
	&\le \frac12C_1^2+C_2\widetilde{C_2}\ln^2{t}+C_2\widetilde{C_2}\ln^2t_0.
	\end{aligned} 
    \end{equation*}
	By multipling the $\frac{\alpha -3}{(\alpha -1)t^2}$, we complete the proof. \qed 
\end{proof}

\begin{thm}\label{thm:WeakParetoPoint} 
	Let \(\alpha>\beta  > 3 \), and let \(x: [t_0, +\infty) \to \mathbb{R}^n\) be a bounded solution of \cref{eq:Cauchy-Problem}. 	 Then, \(x(t)\) converges to a weakly Pareto optimal solution of \cref{eq:MOP}.
\end{thm}

\begin{proof}
	We define the set,
	\[
	V := \{ z \in \mathbb{R}^n : f_i(z) \leq f_i^\infty, \text{ for } i = 1, \dots, m \},
	\]
	where \( f_i^\infty = \lim_{t\to \infty } f_i(x(t))=\lim_{t \to \infty} \mathcal{W}_i(t) \)with existence is guaranteed by \cref{lem:energyfunctionW} and \cref{thm:propositionofmeritfunction}(i).
	
	Since \( x(t) \) is bounded, there exists a limit point \( x^\infty \in \mathbb{R}^n \). Therefore, there exists a sequence \(\{x(t_k)\}_{k \geq 0}\) with \( t_k \to +\infty \) and \( x(t_k) \to x^\infty \) as \( k \to \infty \). Since the objective functions are lower semicontinuous for \( i = 1, \dots, m \), we have
	\[
	f_i(x^\infty) \leq \liminf_{k \to +\infty} f_i(x(t_k)) = \lim_{k \to \infty} f_i(x(t_k))  = f_i^\infty,
	\]
	and hence \( x^\infty \in V \),  which implies that \( V \neq \varnothing \) and any limit point of \( x(t) \) belongs to \( V \).
	
	Let \( z \in V \), and we define \( h_z(t) = \frac{1}{2} \|x(t) - z\|^2 \). The first and second derivatives of \( h_z(t) \) are 
	$\dot{h}_z(t) = \langle x(t) - z, \dot{x}(t) \rangle$ and $
	\ddot{h}_z(t) = \langle x(t) - z, \ddot{x}(t) \rangle + \|\dot{x}(t)\|^2$ 
	for almost all \( t \in [t_0, +\infty) \), respectively. Therefore
	\[
	\ddot{h}_z(t) + \frac{\alpha}{t} \dot{h}_z(t) = \left\langle x(t) - z, \ddot{x}(t) + \frac{\alpha}{t} \dot{x}(t) \right\rangle + \|\dot{x}(t)\|^2.
	\]
	By the definition of  \cref{eq:MAVNG}, 
    there exist $(\theta_1(t),\cdots,\theta_m(t))\in \Delta^m$ such that 
	\begin{equation*}
	\begin{aligned}
	t \ddot{h}_z(t) + \alpha \dot{h}_z(t) 
	&\le t\bigg<z-x(t),\sum_{i=1}^m\theta_i(t)\nabla f_i(x(t))+r(t,x,\dot x)\bigg>+t\|\dot x(t)\|^2.
	\end{aligned}
	\end{equation*}
	 Note that
	\begin{equation*}
    \begin{aligned} 
	\Big<z-x(t),\nabla f_i(x(t))\Big>&\le f_i(z)-f_i(x(t))\le f_i^\infty -f_i(x(t))\le \frac12\|\dot x(t)\|^2,
    \end{aligned} 
	\end{equation*}
	and hence,  we have
	\begin{equation}\label{eq:function-inter}
	\begin{aligned}
	t \ddot{h}_z(t) + \alpha \dot{h}_z(t)&\le \frac32t\|\dot x(t)\|^2+\frac{\alpha -\beta }{t^{p-1}}\|\dot x(t)\|\|x(t)-z\|\\
	&\le \frac32t\|\dot x(t)\|^2+\frac{2R(\alpha -\beta )\widetilde C_1}{t^p}:=g(t),
	\end{aligned}
	\end{equation}
	where $\widetilde C_1$ is defined as in \cref{eq:tilde-C1}. Since \( t \|\dot{x}(t)\|^2 \in L^1([t_0, +\infty),\R) \) and \( p > 1 \), it follows that \( g(t) \in L^1([t_0, +\infty),\R) \). Therefore,  \cref{eq:function-inter} implies 
	$\lim_{t \to \infty} \|x(t) - z\|$ exists by \cref{lem:limitexistlem}.
	Moreover, applying \cref{lem:Opial} and \cref{thm:propositionofmeritfunction}, we get \( \lim_{t \to \infty} x(t) = x^\infty \) and \( \varphi(x(t)) \to 0 \) as \( t \to \infty \). Since \( \varphi(\cdot) \) is lower semicontinuous, we have
	\[
	\varphi(x^\infty) \leq \liminf_{k \to +\infty} \varphi(x(t_k)) = 0.
	\]
	Hence,  \( x^\infty \) is a weakly Pareto optimal solution of \cref{eq:MOP} by \cref{thm:weakpareto}. \qed 
\end{proof}

\section{Algorithm based on \cref{eq:MAVNG}}	    \label{sec:Algorithm}
In this section, we present the discretized iterative scheme for \cref{eq:MAVNG} and an iterative scheme similar to this form, the latter of which can be regarded as an extension of \cref{eq:FISCnes} in multiobjective optimization. 

By directly extending the discretized Lyapunov function from \cite{wang2021search} and integrating the approach from \cite{sonntag2024fastNestrovAlgorithm,tanabe2022globally}, we derive the convergence rate $\varphi(x_k) = {O}(\ln^2 k / k^2)$ (\cref{thm:sequenceConvergence}). Moreover, we demonstrate that all limit points are weakly Pareto optimal solutions (\cref{thm:Convergenceosequence}).
\subsection{Algorithm}

By discretizing \cref{eq:MAVNG}, we can obtain the following iterative scheme. The detailed discussion of the discretization process is deferred to \cref{appendix:subsection-algorithm}.
\begin{equation}\label{eq:Discret-MAVNG}
\left\{\begin{aligned}
y_k&=x_k+\frac{k-1}{k+\alpha -1}\pi_k,\\
x_{k+1}&=y_k-s\frac{k-1}{k+\alpha -1}\proj_{C(y_k)}\left(\frac{1}{s}\pi_k\right).
\end{aligned}\right.
\end{equation}
where $\pi_k =(x_k-x_{k-1})-\frac{\alpha -3}{k -1}\frac{\|x_k-x_{k-1}\|}{\|\proj_{C(x_k)}(0)\|}\proj _{C(x_k)}(0)$.  However, {  since the convergence of \cref{eq:Discret-MAVNG} cannot be effectively established, we adopt an approach similar to that in \cite{sonntag2024fastNestrovAlgorithm} and study the following iterative scheme, which is closely related to \cref{eq:Discret-MAVNG} and exhibits a Nesterov-style form:}
\begin{equation}\label{eq:MFISC-real}
\left\{\begin{aligned}
r_k &= \frac{\|x_k-x_{k-1}\|}{\|\proj_{C(x_k)}(0)\|}\proj _{C(x_k)}(0),\\
\pi_k &=\frac{k-1}{k+\alpha -1}(x_k-x_{k-1})-\frac{\alpha -3}{k+\alpha -1}r_k,\\
y_k&=x_k+\pi_k,\\
x_{k+1}&=y_k-s\proj_{C(y_k)}(\pi_k).
\end{aligned}\right.\tag{MFISC}
\end{equation}
{ This iterative scheme also represents a natural extension of \cref{eq:FISCnes} combined with \cref{eq:AccG} in  multiobjective optimization. Based on the iterative scheme \cref{eq:MFISC-real}, we propose the \cref{algo:MFISCnes2}}. 

{  In \cref{algo:MFISCnes2}, we choose $\|u_k\| < \varepsilon$ as the termination criterion because, in the convex case, it equivalently characterizes weak Pareto optimal solutions.
    The iterative process can be described as follows: we correct the momentum term $x_k - x_{k-1}$ using the normalized steepest descent direction, as shown in \cref{fig:optimization_process2}. Through this correction, the issue of slow convergence caused by the uncertainty of ascent and descent directions in the momentum term can be effectively mitigated, as shown in the numerical experiments.}
\begin{algorithm}[h]
	\caption{Multiobjective Fast Inertial
		Search Direction Correction Algorithm (MFISC)}
	\label{algo:MFISCnes2}
	\begin{algorithmic}[1] 
		\REQUIRE  	 Initial values: $\alpha\geq 3$, $x_0=x_1\in \R^n$ and $u_0=\infty $.\\
		~~~~~~~ Choose $0<s<\frac{1}{L}$  and $k_{\max}$, set $k=1$. \\
		\WHILE{$k<k_{\max}$}
		\STATE Compute $\xi^k\in \argmin_{\xi\in \Delta^m}\frac12\left\|\sum_{i=1}^m\xi_{i}^{k}\nabla f_i(x_k)\right\|^2$ and $u_k =\sum_{i=1}^m \xi_i^k \nabla f_i(x_k)$. 
		\IF{ $\|u_k\|<\varepsilon$}
		\RETURN{$x_k$}	
		\ELSE
		\STATE Set $ \pi_k =\frac{k-1}{k+\alpha -1}(x_k-x_{k-1})-\frac{\alpha -3}{k+\alpha  -1}\frac{\|x_k-x_{k-1}\|}{\|u_k\|}u_k$.\\
		\STATE Set $y_k =x_k+\pi_k$.   
		\STATE Compute $\theta ^k\in \argmin_{\theta\in \R^m } \frac12\left\|s\sum_{i=1}^m\theta_{i}^{k}\nabla f_i(y_k)-\pi_k\right\|^2$. 
		\STATE Compute $ x_{k+1} = y_k -s \sum_{i=1}^m\theta_i^k\nabla f_i(x_k)$ and $k\to k+1$.
		\ENDIF
		\ENDWHILE		
	\end{algorithmic}
\end{algorithm}

\begin{figure}[h]
	\centering
\begin{tikzpicture}[
node distance=1.5cm and 1.8cm,
>=Stealth,
font=\small
]
\node (xk2) {$x_{k-1}$};
\node[right=0.8 cm of xk2] (xk1) {$x_{k}$};
\node[below right=0.5cm and 0.3cm of xk1] (ybar) {$x_{k}-\frac{\alpha -3}{k+\alpha -1}\frac{\|x_k-x_{k-1}\|}{\|\proj_{C(x_k)}(0)\|}\proj_{C(x_k)}(0)$};
\node[right=of ybar] (yk) {$y_{k}$};
\node[above right= 3cm and 0.8cm of yk] (xk) {$x_{k+1}$};

\draw[->] (xk2) -- (xk1);
\draw[->, dashed] (xk1) -- (ybar);
\draw[->, dashed] (ybar) -- (yk);
\draw[->, dashed] (yk) -- (xk);
\draw[->] (xk1) -- (xk);
\end{tikzpicture}
 	\caption{MFISC iteration diagram}
\label{fig:optimization_process2}
\end{figure}
\subsection{Convergence analysis}\label{sec:AlgoforConvergence}

In this subsection, let $\{x_k\}$ and $\{y_k\}$ be the iterations  generated by \cref{algo:MFISCnes2},  we give some notations as follows:
\begin{equation*}
\begin{aligned}
\sigma_k(z)&:=\min_{i=1,\cdots,m} \Big(f_i(x_k)-f_i(z)\Big),\ z\in \R^n,\\
\Delta x_k &:=x_k-x_{k-1}.
\end{aligned}
\end{equation*}

\begin{lem}\label{lem:boundedsequence}
    Let $\{x_k\}$ be the iterations generated by \cref{algo:MFISCnes2}. Then for any $k \ge 0$ and for all $i = 1, \cdots, m$, we have  
    $$
    f_i(x_k) \le f_i(x_0).
    $$
\end{lem}
\begin{proof}
    See \cref{appendix:subsection-Convergence}. \qed 
\end{proof}

\begin{definition}\label{def:discrete-Lyapunov}
    Define the discrete Lyapunov function for $z\in \R^n$:
    \begin{equation}
    \begin{aligned} 
    \mathcal E_z(k)&:=\frac{2(k+\alpha -2)^2s}{\alpha -1}\sigma_k(z)\\&\qquad +2\left\|x_k-z+\frac{k-1}{\alpha -1}\Delta x_k\right\|^2+\frac{(\alpha -3)(k-1)^2}{(\alpha -1)^2}\|\Delta x_k\|^2.
    \end{aligned} 
    \end{equation}
\end{definition}
\begin{lem}\label{eq:Lyapunovk-Lyapunov1}
Let $\mathcal E_z(k)$ for $z\in \R^n$ be as shown in \cref{def:discrete-Lyapunov}, then we have
\begin{equation*}
\mathcal E_z(k)-\mathcal E_z(1 )\le \frac{\alpha -3}{\alpha -1}\phi_{1 }\|\Delta x_{1}\|^2+4\frac{\alpha -3}{\alpha -1}\sum_{j=2}^k \|\Delta x_{j-1}\|\|x_{j-1}-z\| -\frac{2s}{\alpha -1}k\sigma_k(z),
\end{equation*}

\end{lem}
\begin{proof}
    See \cref{appendix:subsection-Convergence}. \qed 
\end{proof} 
\begin{thm}\label{thm:sequenceConvergence}
	Let $\{x_k\}$ be the iterations generated by \cref{algo:MFISCnes2}, then we have
	\begin{equation}
	\varphi(x_k)=O\left(\frac{\ln^2k}{k^2}\right).
	\end{equation}
\end{thm}
\begin{proof} 
	By the definition of \(\mathcal E_z(k)\), \(\mathcal E_z(1)\le 2(\alpha -1)s\varphi(x_0)+2\|x_0-z\|^2:=C(z)\) and \cref{eq:Lyapunovk-Lyapunov1}, we obtain  
	\begin{equation}\label{eq:main:sigmainequatity}
	\begin{aligned}
	&\frac{2s}{\alpha -1}\Big((k+\alpha -2)^2+k\Big)\sigma_k(z)+\frac{(\alpha -3)(k-1)^2}{(\alpha -1)^2}\|\Delta x_k\|^2\\
	\le \ & C(z)+ 
	4\frac{\alpha -3}{\alpha -1}\sum_{j=2}^k\|\Delta x_{j-1}\|\|x_{j-1}-z\|.
	\end{aligned}
	\end{equation}
	Let $z = z_k^{*}$  defined in \cref{coro:sequenceofxk},  $C(z_k^*)\le 2(\alpha -1)s\varphi(x_0)+4R^2:=C$. Then by \cref{assum:levelset,lem:boundedsequence} with $R=\sup_{z \in \mathcal L(f,f(x_0))}\|z\|$, we have
	\begin{equation*} 
	k^2\|\Delta x_k\|^2\le 4\frac{(\alpha -1)^2}{\alpha -3} C+\sum_{j=1}^{k}\frac{32\cdot R(\alpha -1)}j\cdot j\|\Delta x_j\|,\\
	\end{equation*}
	for $k\ge 1$. Further using the \cref{lem:Gronwelldiscrete}, we get:  
	\begin{equation*}
	\begin{aligned}
	k\|\Delta x_k\|&\le \sqrt{4\frac{(\alpha -1)^2}{\alpha -3} C}+\sum_{j=1}^{k}\frac{32\cdot R(\alpha -1)}{j}\le \widetilde C\cdot \ln k,\\
	\end{aligned}
	\end{equation*}
	where $\widetilde C=  \left[{\frac1{\ln 2}\left(\sqrt{4\frac{(\alpha -1)^2}{\alpha -3} C}+32R(\alpha -1)\right)}+32\cdot R(\alpha -1)\right]\cdot \ln k$. By \cref{eq:main:sigmainequatity} and note $\sum_{j=1}^k\frac{\ln j}{j}\le \frac12\ln ^2 k+\frac{\ln 2(1-\ln 2)}{2}$ for $k\ge 2$, we have
	\begin{equation} \label{eq:24}
	\begin{aligned} 
	\frac{2s}{\alpha -1}\Big((k+\alpha -2)^2+k\Big)\sigma_k(z_k^*)&\le C+
	4\frac{\alpha -3}{\alpha -1}\widetilde C\cdot 2R\sum_{j=1}^k\frac{\ln j}{j}\\
	&\le  \left[\frac{ C }{(\ln 2)^2} + \frac{8 \widetilde C R (\alpha - 3) }{ (\alpha - 1) \ln 2 }\right]\cdot \ln ^2 k.
	\end{aligned}
	\end{equation}
    Since $\varphi(x_k)=\sigma_k(z_k^*)$, multiplying both sides of \cref{eq:24} by $\frac{\alpha -1}{2s}\cdot \frac{1}{(k+\alpha -2)^2+k}$ yields the conclusion of the theorem.\qed 
\end{proof}
\begin{thm}\label{thm:Convergenceosequence}
	Let $\{x_k\}$ be the iterations generated by \cref{algo:MFISCnes2}, then all limit points of ${x_k}$ are weakly Pareto optimal solutions of \cref{eq:MOP}.
\end{thm}
\begin{proof}
	By \cref{thm:sequenceConvergence} and \cref{thm:weakpareto} , the result follows immediately.\qed 
\end{proof}

\section{Numerical experiments}
\label{sec:Numercial}
In this section, we experimentally verify the theoretical results of the \cref{eq:MAVNG} trajectory solution as well as the effectiveness of the algorithm  \cref{eq:MFISC}. All numerical experiments were performed in the MATLAB 2021a environment on a personal computer with an Intel(R) Core(TM) i5-8300H CPU @ 2.30GHz   2.30 GHz processor and 8GB of RAM. All involved quadratic subproblems are solved using the Frank-Wolfe method.
\subsection{Examples for \cref{eq:MAVNG}}
{ 
In this subsection, we use two examples from \cite{sonntag2024fastgradientflow} to demonstrate the properties of our equation proved earlier, including convergence to weakly Pareto optimal solutions and the convergence rate \(\varphi(x(t)) = O(1/t^2)\). These two examples are quadratic programming and non-quadratic programming, respectively, and their specific forms can be found in \cref{appendix:test_Problem}. 

For \cref{eq:MAVNG}, we set $p=1.01$ and $\beta = 3.01$. For different parameter selections $\alpha \in \{5, 10, 50, 100\}$, we conducted comparative experiments with \cref{eq:MAVD} under the same $\alpha$ selection. The trajectories were approximated using discretization with a stepsize of $h = 10^{-3}$, $t_0 = 1$, and $t \approx t_k = t_0 + kh$. The experimental results for quadratic programming and non-quadratic programming are shown in \cref{fig:figuretrajectories1} and \cref{fig:figuretrajectories2}, respectively. The main conclusions of the numerical experiments are as follows:

\begin{itemize}
\item[$\rm (i)$] For \cref{eq:MAVNG}, we observed that on these two problems, the trajectory solutions converge to weakly Pareto optimal solutions, and under the characterization of the merit function $\varphi(x(t))$, the convergence rate achieves $O(1/t^2)$, which is consistent with the theoretical results provided in \cref{sec:Asymptotic}.

\item[$\rm (ii)$] For \cref{fig:figuretrajectories1}, we found that on this quadratic programming problem, the trajectory of \cref{eq:MAVNG} does not differ significantly from that of \cref{eq:MAVD}, and the final convergence results are very similar ((a)--(d)). Moreover, as $\alpha$ increases, the trajectory solutions of \cref{eq:MAVNG} exhibit faster convergence. Both equations demonstrate the following characteristic: the trend of the merit function value is robust for larger $\alpha$, as we did not observe significant differences between them in subfigures (g) and (h). 

\item[$\rm (iii)$] For \cref{fig:figuretrajectories2}, we observed that on the non-quadratic programming problem, the trajectory of \cref{eq:MAVNG} also does not differ significantly from that of \cref{eq:MAVD}. However, for larger $\alpha$, the function value decreases slowly along the trajectory of \cref{eq:MAVD}. In the case of \cref{eq:MAVNG}, after a period of slow change, the function value continues to decrease fast.  
\end{itemize}
Based on the above three points, we conclude that \cref{eq:MAVNG} is an effective and competitive continuous algorithm. Additionally, the advantages exhibited by \cref{eq:MAVNG} for larger $\alpha$ provide guidance for parameter selection in discrete algorithms.}
\begin{figure}[H]
	\centering
	\includegraphics[width=0.8\linewidth]{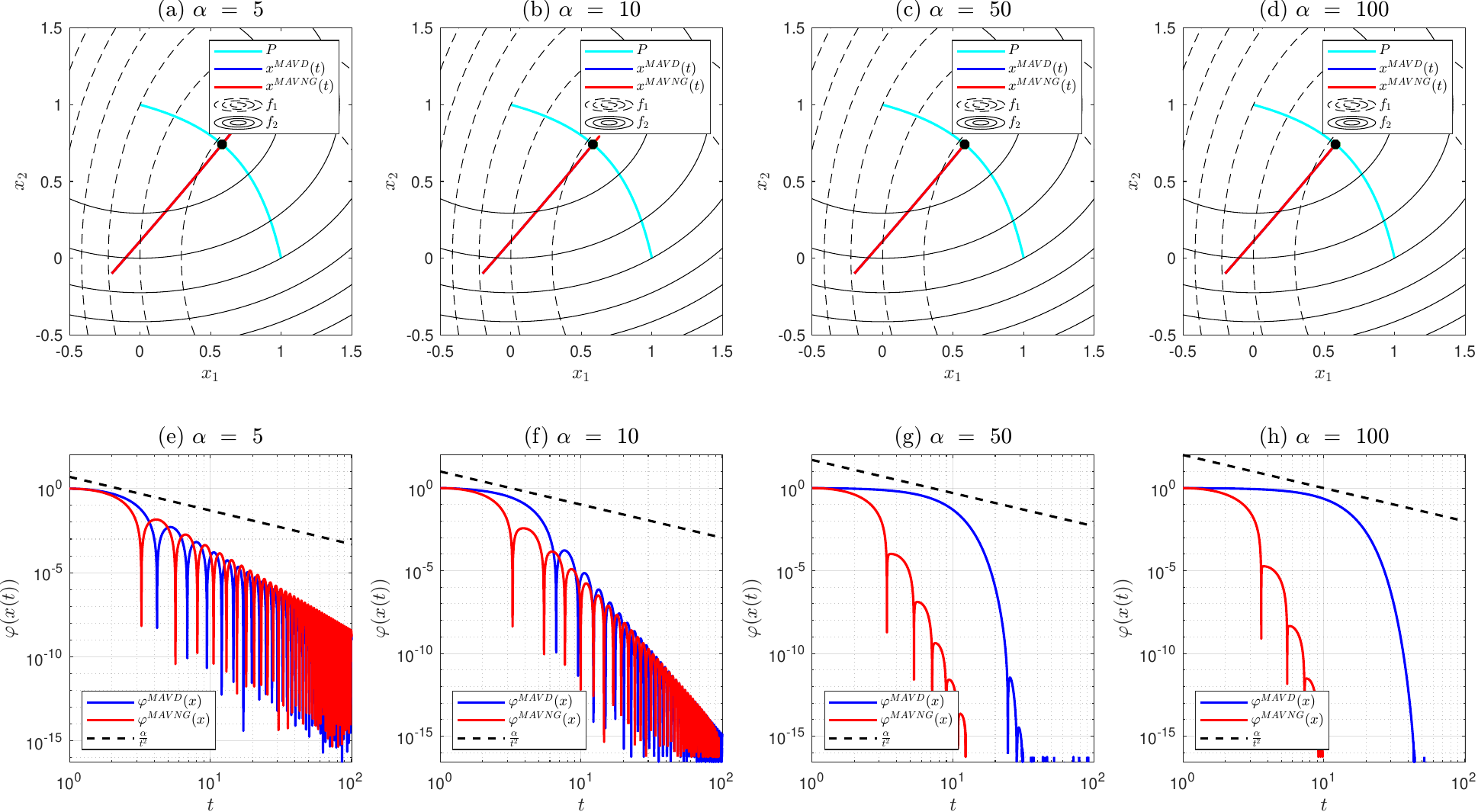}
	\caption{For the quadratic programming problems, the trajectories and changes in function values of $\rm MAVNG$ and $\rm MAVD$. The red line corresponds to MAVNG, and the blue line corresponds to MAVD.}
	\label{fig:figuretrajectories1}
\end{figure}
\begin{figure}[H]
	\centering
	\includegraphics[width=0.8\linewidth]{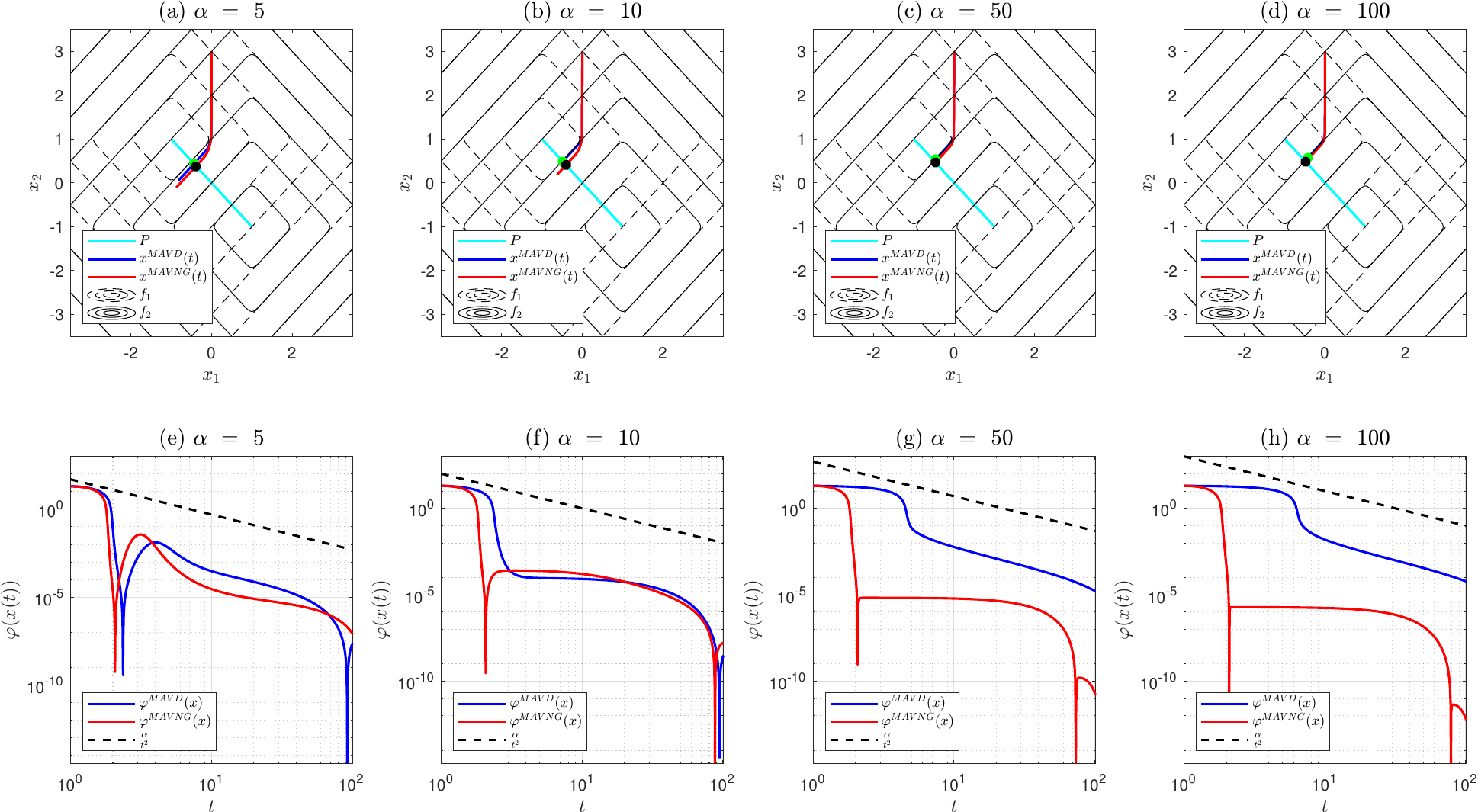}
	\caption{For the non-quadratic programming problems, the trajectories and changes in function values of $\rm MAVNG$ and $\rm MAVD$. The red line corresponds to MAVNG, and the blue line corresponds to MAVD.}
	\label{fig:figuretrajectories2}
\end{figure}

\subsection{Examples for \cref{eq:MFISC-real}}
{ In this section, we use the discrete algorithm \cref{eq:MFISC-real} to solve multiple multiobjective optimization problems and conduct comparative experiments with AccG \cite{sonntag2024fastgradientflow} and the steepest descent method MGSD \cite{fliege2000steepest}. The test problems are summarized in \cref{tab:test_problems}, and details can be found in \cref{appendix:test_Problem}. Based on the experimental conclusions of the continuous algorithm in the previous subsection, which indicate that \cref{eq:MAVNG} exhibits strong competitiveness when $\alpha$ is selected as a larger value, we choose the parameter $\alpha = 50$ in \cref{eq:MFISC}. Additionally, the termination criterion tolerance is set to $\varepsilon = 1e-4$ and the maximum number of iterations is $k_{\max} = 1000$. Based on these parameter selections, we conduct the following two  experiments:
\begin{itemize}
    \item[$\rm (i)$] According to the initial point selection ranges for each test problem in \cref{tab:test_problems}, we randomly select 100 initial points for experimentation. For stepsizes $s \in \{0.01, 0.05, 0.1\}$, we compute the average CPU runtime, average number of iterations, and problem-solving ratio for the three algorithms. The problem-solving ratio is defined as the ratio of the number of points that meet the termination criterion within the maximum number of iterations to the total 100 points. We record the most advantageous results for each algorithm under different stepsize conditions, as shown in \cref{tab:result}.\\
    \item[$\rm (ii)$] For three convex problems (JOS1a, SDa, TOI4a) and three non-convex problems (DD1, KW2, LTY3c), we select 500 initial points and apply $\rm MFISC$ and $\rm AccG$ to obtain the image sets at the final iteration points. The image sets generated for these two types of problems are plotted in \cref{fig:figure_convex_pareto_front,fig:figure_nonconvex_pareto_front}, respectively.
\end{itemize}}

\begin{table}[h]
    \centering
    \caption{Description of all test problems used in numerical experiments. The LTY3, DD1 and KW2 are nonconvex probelms.}
        \setlength{\tabcolsep}{12pt}
    \label{tab:test_problems}
    \begin{tabular}{c c c c c c }
        \toprule
        {Problem} & \({n}\) & \({m}\) & \({x_L}\) & \({x_U}\) & {Ref} \\
        \midrule
        JOS1a        & 5 & 2 & \((-10,\dots,-10)\) & \((10,\dots,10)\)      & \cite{mita2019nonmonotone} \\
        JOS1b        & 20 & 2 & \((-10,\dots,-10)\) & \((10,\dots,10)\)      & \cite{mita2019nonmonotone} \\
        JOS1c        & 50 & 2 & \((-10,\dots,-10)\) & \((10,\dots,10)\)      & \cite{mita2019nonmonotone} \\
        JOS1d        & 100 & 2 & \((-10,\dots,-10)\) & \((10,\dots,10)\)      & \cite{mita2019nonmonotone} \\
        SDa        & 4 & 2 & \((1,\sqrt 2,\sqrt 2.1)\)       & \((3,3,3,3)\)              & \cite{mita2019nonmonotone} \\
        SDb        & 20 & 2 & \((1,\sqrt 2,\dots,\sqrt 2.1)\)       & \((3,3,\dots ,3,3)\)              & \cite{mita2019nonmonotone} \\
        SDc        & 50 & 2 & \((1,\sqrt 2,\dots,\sqrt 2.1)\)       & \((3,3,\dots,3,3)\)              & \cite{mita2019nonmonotone} \\
        TOI4a       & 4 & 2 & \((-2,-2,-2,-2)\)         & \((5,5,5,5)\)              & \cite{mita2019nonmonotone} \\
        TOI4b       & 40 & 2 & \((-2,\dots,-2)\)         & \((5,\dots,5)\)              & \cite{mita2019nonmonotone} \\
        TOI4c       & 100 & 2 & \((-2,\dots,-2)\)         & \((5,\dots,5)\)              & \cite{mita2019nonmonotone} \\
        FDSa        & 5 & 3 & \((-2,\dots,-2)\)   & \((2,\dots,2)\)        & \cite{mita2019nonmonotone} \\
        FDSb        & 20 & 3 & \((-2,\dots,-2)\)   & \((2,\dots,2)\)        & \cite{mita2019nonmonotone} \\
        FDSc        & 50 & 3 & \((-2,\dots,-2)\)   & \((2,\dots,2)\)        & \cite{mita2019nonmonotone} \\
        FDSd        & 100 & 3 & \((-2,\dots,-2)\)   & \((2,\dots,2)\)        & \cite{mita2019nonmonotone} \\
        LTY1a        & 50 & 3 & \((-15,\dots,-15)\)      & \((15,\dots,15)\)             & \cite{luo2025accelerated} \\
        LTY1b        & 100 & 3 & \((-15,\dots,-15)\)      & \((15,\dots,15)\)             & \cite{luo2025accelerated} \\
        LTY1c        & 200 & 3 & \((-15,\dots,-15)\)      & \((15,\dots,15)\)             & \cite{luo2025accelerated} \\
        LTY1d        & 300 & 3 & \((-15,\dots,-15)\)      & \((15,\dots,15)\)             & \cite{luo2025accelerated} \\
        LTY2a        & 10 & 3 & \((-2,\dots,-2)\)      & \((2,\dots,2)\)             & \cite{luo2025accelerated} \\
        LTY2b        & 20 & 3 & \((-2,\dots,-2)\)      & \((2,\dots,2)\)             & \cite{luo2025accelerated} \\
        LTY2c        & 30 & 3 & \((-2,\dots,-2)\)      & \((2,\dots,2)\)             & \cite{luo2025accelerated} \\
        LTY3a       & 10 & 2 & \((-2,\dots,-2)\)      & \((2,\dots,2)\)             & \cite{luo2025accelerated} \\
        LTY3b        & 50 & 2 & \((-2,\dots,-2)\)      & \((2,\dots,2)\)             & \cite{luo2025accelerated} \\
        LTY3c        & 100 & 2 & \((-2,\dots,-2)\)      & \((2,\dots,2)\)             & \cite{luo2025accelerated} \\
        DD1       & 5 & 2 & \((-20,\cdots,-20)\)           & \((20,\cdots,20)\)              & \cite{mita2019nonmonotone}\\
        KW2       & 2 & 2 & \((-3,-3)\)           & \((3,3)\)              & \cite{mita2019nonmonotone} \\
        \bottomrule
    \end{tabular}
\end{table}

\begin{table}[htbp]
    \centering
    \caption{The most advantageous average CPU time(s), average number of iterations, and problem-solving ratio of MFISC, AccG, and MGSD at stepsizes $s \in \{0.01, 0.05, 0.1\}$. "Most advantageous" refers to the result with the shortest CPU time under the condition of achieving the highest problem-solving ratio.}
    \setlength{\tabcolsep}{3.8pt}
    \label{tab:result}
    \begin{tabular}{llllllllll}
        \toprule
        \multirow{2}{*}{Problem} & \multicolumn{3}{c}{MFISC} & \multicolumn{3}{c}{AccG} & \multicolumn{3}{c}{MGSD} \\
        \cmidrule(lr){2-4} \cmidrule(lr){5-7} \cmidrule(lr){8-10}
        & Time & Iter & Ratio & Time & Iter & Ratio & Time & Iter & Ratio \\
        \midrule
        {JOS1a} & \textbf{0.0031} & 28.28 & 1.00 & 0.0088 & 185.12 & 1.00 & 0.0059 & 264.29 & 1.00 \\
        {JOS1b} & \textbf{0.00078} & 35.28 & 1.00 & 0.0039 & 264.93 & 1.00 & 0.015 & 999.18 & 0.12 \\
        {JOS1c} & \textbf{0.0016} & 63.10 & 1.00 & 0.0028 & 210 & 1.00 & 0.013 & 1000 & 0.00 \\
        {JOS1d} &\textbf{0.0031} & 97.01 & 1.00 & 0.0059 & 438 & 1.00 & 0.013 & 1000 & 0.00 \\
        {SDa} & \textbf{0.0030} & 24.10 & 1.00 & 0.0088 & 133.79 & 1.00 & 0.0059 & 270.25 & 1.00 \\
         {SDb} &\textbf{0.00078} & 27.28 & 1.00 & 0.0033 & 181.78 & 1.00 & 0.0044 & 263.12 & 1.00 \\
         {SDc} & \textbf{0.00078} & 24.55 & 1.00 & 0.0034 & 188.27 & 1.00 & 0.0050 & 266.06 & 1.00 \\
         {TOI4a} & {0.0036} & 27.36 & 1.00 & \textbf{0.0028} & 41.66 & 1.00 & 0.0033 & 48.95 & 1.00 \\
         {TOI4b} & \textbf{0.0013} & 32.67 & 1.00 & 0.0016 & 66.99 & 1.00 & 0.0014 & 59.60 & 1.00 \\
         {TOI4c} & \textbf{0.0011} & 31.73 & 1.00 & 0.0016 & 66.16 & 1.00 & 0.0013 & 59.22 & 1.00 \\
         {FDSa} & \textbf{0.089} & 32.64 & 1.00 & 0.10 & 203.75 & 1.00 & 0.73 & 726.67 & 0.67 \\
         {FDSb} & \textbf{0.030} & 21.08 & 1.00 & 0.28 & 442.27 & 1.00 & 0.24 & 958.37 & 0.12 \\
         {FDSc} & \textbf{0.020} & 24.07 & 1.00 & 0.57 & 675.51 & 1.00 & 0.027 & 1000 & 0.00 \\
         {FDSd} & \textbf{0.013} & 53.78 & 1.00 & 0.43 & 767.29 & 0.88 & 0.033 & 1000 & 0.00 \\
         {LTY1a} & \textbf{0.094 }& 55.23 & 1.00 & 0.48 & 701.86 & 1.00 & 0.49 & 999.87 & 0.01 \\
         {LTY1b} & \textbf{0.67} & 336.54 & 0.87 & 0.67 & 990.24 & 0.15 & 0.43 & 1000 & 0.00 \\
         {LTY1c} & \textbf{2.87 }& 396.63 & 0.82 & 1.80 & 1000 & 0.00 & 2.5 & 1000 & 0.00 \\
         {LTY1d} & \textbf{4.91} & 475.14 & 0.57 & 1.60 & 670.66 & 0.33 & 3.8 & 1000 & 0.00 \\
         {LYT2a} & {0.35} & 133.33 & 0.95 & \textbf{0.27} & 387.13 & 1.00 & 0.86 & 988.68 & 0.07 \\
         {LYT2b} & {0.20} & 133.21 & 0.99 & \textbf{0.62} & 606.69 & 1.00 & 0.29 & 1000 & 0.00 \\
         {LYT2c} & \textbf{0.32} & 222.86 & 1.00 & 0.45 & 418.51 & 1.00 & 0.19 & 1000 & 0.00 \\
         {LTY3a} & {0.0017} & 20.35 & 1.00 & \textbf{0.0013} & 62.37 & 1.00 & 0.0020 & 103.15 & 1.00 \\
         {LTY3b} & \textbf{0.00047} & 19.53 & 1.00 & \textbf{0.00047 }& 29.13 & 1.00 & \textbf{0.00047} & 30.86 & 1.00 \\
         {LTY3c} & {0.00063} & 15.76 & 1.00 & \textbf{0.00031} & 22.72 & 1.00 & \textbf{0.00031} & 13.76 & 1.00 \\
         {DD1} & \textbf{0.0028} & 29.96 & 1.00 & 0.0097 & 260.70 & 1.00 & 0.040 & 434.50 & 0.98 \\
         {KW2} & {0.0014} & 22.91 & 1.00 & \textbf{0.00078} & 60.86 & 1.00 & 0.0023 & 200.48 & 0.84 \\
        \hline
    \end{tabular}
\end{table}

\begin{figure}[H]
    \centering
    \includegraphics[width=0.8\linewidth]{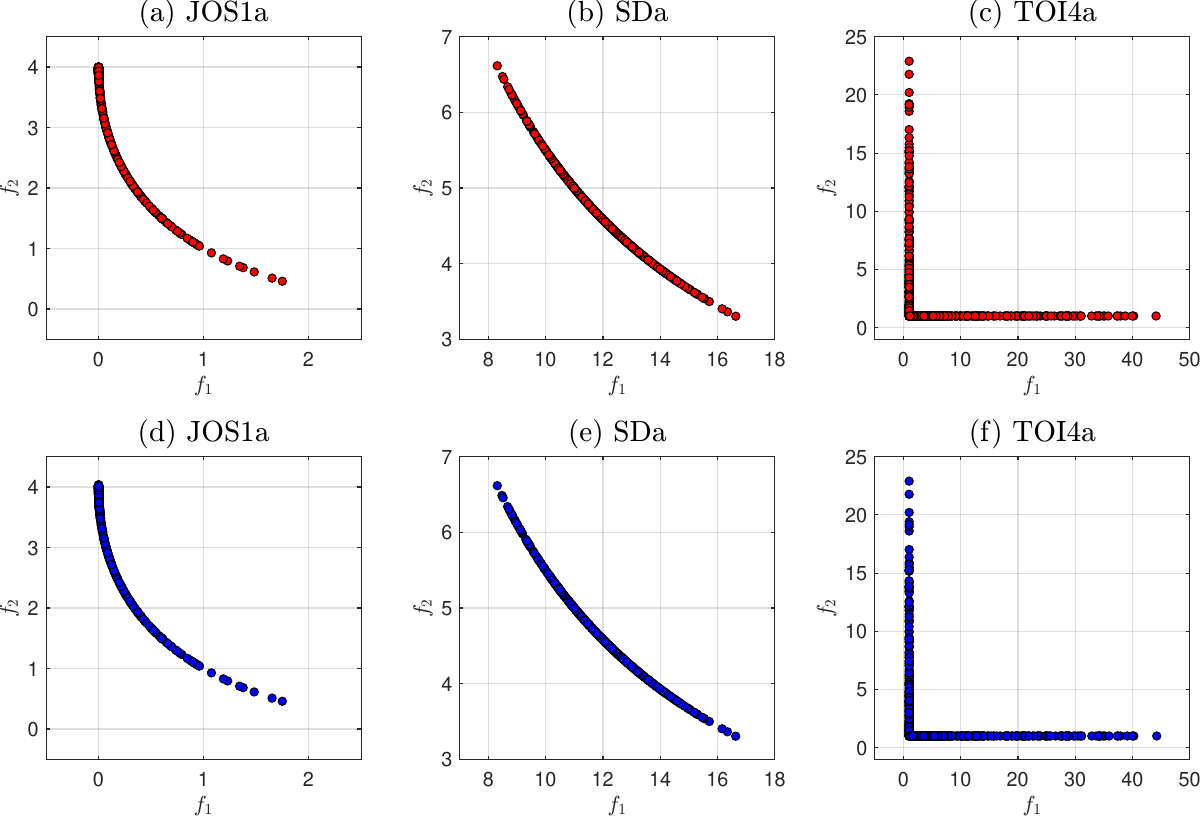}
    \caption{The image sets generated by the two comparative algorithms for the convex problems JOS1a, SDa and TOI4a. (a)--(c) correspond to MFISC; (d)--(f) correspond to AccG.}
    \label{fig:figure_convex_pareto_front}
\end{figure}

\begin{figure}[H]
    \centering
    \includegraphics[width=0.8\linewidth]{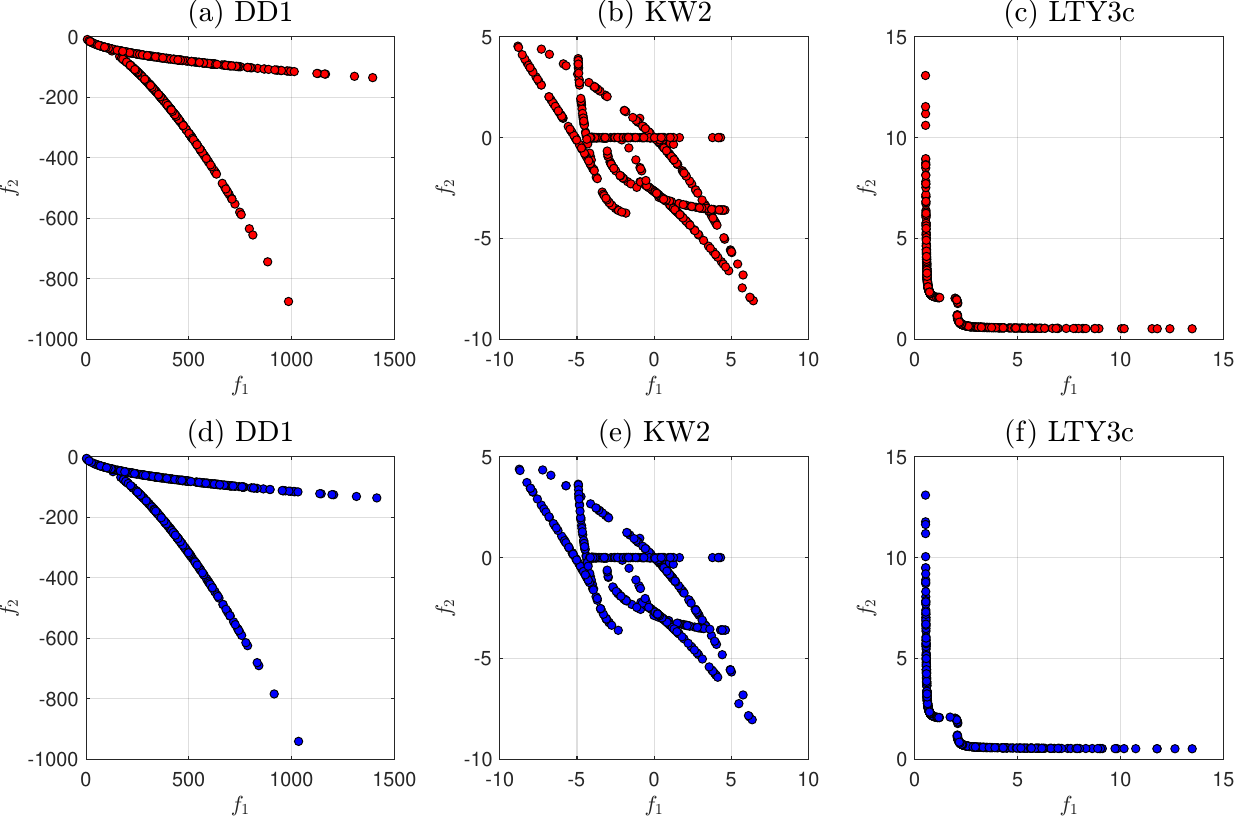}
    \caption{The image sets generated by the two comparative algorithms for the non-convex problems  DD1, KW2 and LTY3c. (a)--(c) correspond to MFISC; (d)--(f) correspond to AccG.}
    \label{fig:figure_nonconvex_pareto_front}
\end{figure}

{ 
The main conclusions of the numerical experiments on the discrete algorithm are as follows:

\begin{itemize}
    \item[$\rm (i)$] From the experimental results in \cref{tab:result}, it can be observed that our algorithm \cref{eq:MFISC-real} generally outperforms  $\rm AccG$ and $\rm MGSD$ in terms of CPU time, number of iterations, and problem-solving ratio. It is worth mentioning that for the convex problem LTY2, when the dimensionality is small, AccG performs better than $\rm MFISC$ in terms of time, but underperforms when the dimensionality is larger. For non-convex problems, AccG slightly outperforms MFISC overall. However, it should be noted that the advantages of AccG for LTY3 are all results under the stepsize selection of $s = 0.1$, which is often not chosen in practice. The reason is that a larger stepsize is more likely to violate the stepsize range required by the algorithm.
    \item[$\rm (ii)$] As shown in \cref{fig:figure_convex_pareto_front,fig:figure_nonconvex_pareto_front}, the final image set generated by MFISC exhibits low discrepancy compared to that of AccG. In fact, in the convex case, this image set represents the weak Pareto front, and \cref{fig:figure_convex_pareto_front} further validates the results we discussed in \cref{sec:Algorithm}.
\end{itemize}
Overall, MFISC demonstrates competitiveness as a discretization algorithm. For the cases where its performance is relatively poor, we attribute this to the fact that the algorithm involves solving two quadratic programming subproblems during each iteration, which is an aspect that requires further improvement in future work.}

\section{Conclusion}
In this work, we interpret and extend the dynamical system and algorithm proposed by Wang et al.\cite{wang2021search} for single-objective optimization to multiobjective settings, demonstrating strong competitiveness. Under appropriate parameter choices, we prove that the system achieves a convergence rate of $O(1/t^2)$ towards weakly Pareto optimal solutions for smooth convex multiobjective optimization problems. For the algorithm, we establish a convergence rate of $O(\ln^2 t / t^2)$. Numerical experiments indicate that this extension offers distinct advantages in multiobjective optimization.  { Our algorithm involves solving two quadratic programs per iteration, which leads to relatively high computational costs in certain problems. In future work, we will refine this aspect to enhance efficiency.} 

In addition, it is interesting to either extend existing multiobjective algorithms into continuous forms or accelerate the convergence of gradient flows. In the future, we will investigate multiobjective balanced gradient flows and gradient flows with time scaling. The former extends a class of multiobjective algorithms designed for solving imbalanced problems into continuous forms, while the latter can further improve the theoretical convergence rates of dynamical systems. Preliminary research findings can be found in \cite{yin2025multiobjective,yin2025multiobjective2}.

\begin{acknowledgements}
    The authors would like to express sincere gratitude to Hao Luo (\url{luohao@cqnu.edu.cn}) for his crucial guidance on the key theoretical methodology of this paper; to Chengzhi Huang and Zhuoxin Fan for their contributions to enhancing the quality of the manuscript; and to Wenzhe Zhao,  Hua Liu and Jiaxin Li for their invaluable assistance in solidifying the foundational knowledge of the first author Yingdong Yin.
\end{acknowledgements}

\begin{appendices}
	\appendix
	\section{Auxiliary lemmas}
    \begin{lem}[Existence theorem]\label{lem:exist}
        Let \(\mathcal{X}\) be a real Hilbert space, and let \(\Omega \subset \mathbb{R} \times \mathcal{X}\) be an open subset containing \((t_0, x_0)\). Let \(G\) be an upper semicontinuous map from \(\Omega\) into the nonempty closed convex subsets of \(\mathcal{X}\). We assume that \((t, x) \mapsto \proj_{G(t,x)}(0)\) is locally compact. Then, there exists \(T > t_0\) and an absolutely continuous function \(x\) defined on \([t_0, T]\), which is a solution to the differential inclusion
        \begin{equation*}
        \dot{x}(t) \in G(t, x(t)), \quad x(t_0) = x_0.
        \end{equation*}
    \end{lem}
    \begin{proof}
        See \cite[Theorem 3.1]{aubin2009differential}. \qed 
    \end{proof} 
    
\begin{lem}[Opial's lemma]	\label{lem:Opial}	
    Let \( V \subseteq \mathbb{R}^n \) be a nonempty subset, and let \( x: [t_0, +\infty) \to \mathbb{R}^n \). Assume that \( x \) satisfies the following conditions,
    \begin{itemize}
        \item[\(\rm (i)\)] Every limit point \( x^* \) of \( x \) belongs to \( V \), i.e., there exists a sequence \(\{t_k\} \subseteq \mathbb{R}^n\) such that \( x^* = \lim_{k \to \infty} x(t_k) \) exists and belongs to \( V \).
        
        \item[\(\rm (ii)\)] For every \( z \in V \), \( \lim_{t \to \infty} \|x(t) - z\| \) exists.
    \end{itemize}
    Then, \( x(t) \) converges to an point \( x^\infty \in V \) as \( t \to \infty \).
\end{lem}
\begin{proof}
    See \cite[Lemma 2.1]{Attouch2014}. 
\end{proof}

\begin{lem}\label{lem:limitexistlem}
    Let \(t_0 > 0\), and let \(h: [t_0, +\infty) \to \mathbb{R}\) be a continuously differentiable function with a lower bound. Assume
    \[
    t \ddot h(t) + \alpha \dot h(t) \le g(t),
    \]
    for some \(\alpha > 1\) and almost all \(t \in [t_0, +\infty)\), where \(g \in L^1([t_0, +\infty))\) is a nonnegative function. Then, \(\lim_{t \to +\infty} h(t)\) exists.
\end{lem}
\begin{proof}
    See \cite[Lemma 5.9]{attouch2018fast}.\qed
\end{proof} 
\begin{lem}\label{lem:lemmaA4}
    Let $C\subseteq \R^n$ be a convex and closed set and $\eta \in \R^n$ a fixed vector. Then, $\xi \in \R^n$ is a solution to the problem
    \begin{equation*}
    \textit{Find } \xi \in \mathbb{R}^n \textit{ such that } \eta = \proj_{C+\xi}(0),
    \end{equation*}
    if and only if it has the form $\xi =\eta -\mu$, where $\mu$ is a solution to the constrained optimzation problem $\min_{\mu\in C}\langle \mu, \eta \rangle $. 
\end{lem}
\begin{proof}
    See \cite[Lemma A.1]{sonntag2024fastNestrovAlgorithm}.\qed 
\end{proof}
\begin{lem}\label{lem:projection_equation} Let \( C \subseteq \mathbb{R}^n \) be a convex and closed set, and let \( a > 0 \), \( v \in \mathbb{R}^n \) be fixed. Then the problem  
\[
\textit{Find } \xi \in \mathbb{R}^n \textit{ such that } -a(\xi + v) = \proj_{C+\xi}(0),
\]  
has a unique solution, given by
\[
\xi = -\left( \frac{1}{1+a} \proj_{C}(v) + \frac{a}{1+a} v \right).
\]
	\end{lem}
\begin{proof}
    See \cite[Lemma A.2]{sonntag2024fastNestrovAlgorithm}.\qed 
\end{proof}
	\begin{lem}\label{lem:lemmaA2} Let \( t_0 \in \mathbb{R} \) and \( T \in (t_0, +\infty) \). Let \( a \in [0, +\infty) \) and \( g \in L^1([0, T], \mathbb{R}) \) with \( g(t) \geq 0 \) almost everywhere for \( t \in [0, T] \). Suppose \( h \in C([0, T], \mathbb{R}) \) satisfies  
		\[
		h(t) \leq a + \int_{t_0}^t g(s)h(s) \, ds, \quad \text{for all } t \in [t_0, T].
		\]  
		Then  
		\[
		h(t) \leq a \cdot \exp\left( \int_{t_0}^t g(s) \, ds \right) ,\quad \text{for all } t \in [t_0, T].
		\]
	\end{lem}
\begin{proof}
    See \cite[Lemma A.1]{attouch2015multiibjective}.\qed 
    \end{proof} 
	\begin{lem}\label{lem:lemmaA3} Let \( t_0 \in \mathbb{R} \) and \( T \in (t_0, +\infty) \). Let \( a \in [0, +\infty) \), and \( g \in L^1([t_0, T], \mathbb{R}) \) with \( g(t) \geq 0 \) almost everywhere for \( t \in [t_0, T] \). Suppose \( h \in C([t_0, T], \mathbb{R}) \) satisfies  
		\[
		\frac{1}{2} h^2(t) \leq \frac{\alpha^2}{2} + \int_{t_0}^t g(s)h(s), \ ds \quad \text{for all } t \in [t_0, T].
		\]  
		Then  
		\[
		|h(t)| \leq a + \int_{t_0}^t g(s) \, ds \quad \text{for } t \in [t_0, T].
		\]
	\end{lem}
\begin{proof}
    \cite[Lemma 5.13]{attouch2018fast}.\qed 
\end{proof}
	\begin{lem}\label{lem:Gronwelldiscrete}
		
		Let $\{a_k\}$ be a nonnegative sequence such that
		\[
		a_k^2 \leq c^2 + \sum_{j=1}^{k} \beta_j a_j \quad \text{for all } k \geq 1,
		\]
		where $\{\beta_j\}$ is a nonnegative sequence and $c \geq 0$. Then, 
		\[
		a_k \leq c + \sum_{j=1}^k \beta_j\quad \text{for all } k \geq 1,
		\]
	\end{lem}
	\begin{proof}
		For $k \geq 1$, let $A_k := \max_{1 \leq m \leq k} a_m$. Then for $1 \leq m \leq k$, we have
		\[
		a_m^2 \leq c^2 + \sum_{j=1}^{m} \beta_j a_j \leq c^2 + \sum_{j=1}^{k} \beta_j a_j \leq c^2 + A_k \sum_{j=1}^k \beta_j.
		\]
		Thus,
		\[
		A_k^2 \leq c^2 + A_k \sum_{j=1}^k \beta_j.
		\]
		Therefore, by solving the quadratic inequality, we obtain
		\[
		a_k \leq A_k \leq \frac{ \sum_{j=1}^k \beta_j + \sqrt{ \left( \sum_{j=1}^k \beta_j \right)^2 + 4c^2 } }{2} \leq  c + \sum_{j=1}^k \beta_j.
		\]
		The proof is complete.\qed 
	\end{proof}
	
    \section{Proof of \cref{lem:LyapunovInequal}}\label{appendix:ProofofLyapunov}
    \begin{definition}
         Let $x:[t_0,+\infty )\to \R^n$ be a global solution of \cref{eq:Cauchy-Problem}. For the objective function $f_i$, $i=1,\cdots,m$, we define the following Lyapunov function  
         $$
         \begin{aligned}  
         \mathcal E_z^i(t)&=\frac{t^2}{2(\alpha -1)}\Big(f_i(x(t))-f_i(z)\Big)\\  
         &\qquad +\frac{(\alpha -3)t^2}{4(\alpha -1)^2}\|\dot x(t)\|^2+\frac12\left\|x(t)-z+\frac{t}{\alpha -1}\dot x(t)\right\|   
         \end{aligned}  
         $$
         Recalling \cref{def:discrete-Lyapunov}, we have the following Lyapunov function  
         $$
         \begin{aligned}  
         \mathcal E_z(t)&=\min_{i=1,\cdots,m}\mathcal E_z^i(t)=\frac{t^2}{2(\alpha -1)}\min_{i=1,\cdots,m}\Big(f_i(x(t))-f_i(z)\Big)\\  
         &\qquad +\frac{(\alpha -3)t^2}{4(\alpha-1)^2}\|\dot x(t)\|^2+\frac12\left\|x(t)-z+\frac{t}{\alpha -1}\dot x(t)\right\|  
         \end{aligned}  
         $$
    \end{definition} 
\begin{lem}  Let $x:[t_0,+\infty )\to \R^n$ be a global solution of \cref{eq:Cauchy-Problem}. Then  
 $$
 \begin{aligned}  
 \frac{d}{dt}\mathcal E_z^i(t)&\le \frac{(3-\beta)t}{2(\alpha -1)}\|\dot x(t)\|^2-\frac{t}{\alpha -1}\min_{i=1,\cdots,m}\Big(f_i(x(t))-f_i(z)\Big)\\  
 &\qquad +\frac{t}{\alpha -1}\Big(f_i(x(t))-f_i(z)\Big)+\frac{\alpha -\beta}{(\alpha -1)t^{p-1}}\|\dot x(t)\|\|x(t)-z\|  
 \end{aligned}  
 $$
\end{lem}
{\it Proof\ }By directly computing $\frac{d}{dt}\mathcal E_z^i(t)$,  
\begin{equation*}
\begin{aligned}
\frac{d}{dt}\mathcal E_z^i(t)&=\frac{(\alpha -3)t}{2(\alpha -1)^2}\|\dot x(t)\|^2+\frac{(\alpha -3)t^2}{2(\alpha -1)^2}\Big<\dot x(t),\ddot x(t)\Big>\\
&\quad +\left<x-z+\frac{t}{\alpha -1}\dot x(t),\frac{\alpha}{\alpha -1}\dot x(t)+\frac{t}{\alpha -1}\ddot x(t)\right>\\
&\quad +\frac{t}{\alpha -1}\Big(f_i(x)-f_i(z)\Big)+\frac{t^2}{2(\alpha -1)}{\Big<\nabla f_i(x(t)),\dot x(t)}\Big>\\
&=\frac{3t}{2(\alpha -1)}\|\dot x(t)\|^2+\frac{t^2}{2(\alpha -1)}\bigg(\Big<\dot x(t),\ddot x(t)\Big>+\Big<\nabla f_i(x(t)),\dot x(t)\Big>\bigg)\\
&\quad +\frac{t}{\alpha -1}\left<z-x(t),\sum_{i=1}^m\theta_i(t)\nabla f_i(x(t))+\frac{\alpha -\beta }{t^p}\frac{\|\dot x(t)\|}{\|\proj_{C(x(t))}(0)\|}\proj_{C(x)}(0)\right>\\
&\quad +\frac{t}{\alpha -1}\Big(f_i(x(t))-f_i(z)\Big)\\
&\le \frac{(3-\beta)t}{2(\alpha -1)}\|\dot x(t)\|^2-\frac{t}{\alpha -1}\min_{i=1,\cdots, m}\Big(f_i(x(t))-f_i(z)\Big)\\
&\quad +\frac{t}{\alpha -1}\Big(f_i(x(t))-f_i(z)\Big)+\frac{\alpha -\beta }{(\alpha -1)t^{p-1}}\|\dot x(t)\|\|x(t)-z\|. 
\end{aligned}
\end{equation*}
The proof is complete.\qed 
    \begin{lem}\label{lem:diffequaldiffi}
    Let \(\{h_i\}_{i=1,\cdots,m}\) be a set of continuously differentiable functions, \(h_i: [t_0, +\infty) \to \mathbb{R}\). Define \(h: [t_0, +\infty) \to \mathbb{R}\), \(t \mapsto h(t) := \min_{i=1,\cdots,m} h_i(t)\). Then, the following holds:
    \begin{itemize}
        \item[\(\rm (i)\)] \(h\) is differentiable almost everywhere on \([t_0, +\infty)\);
        
        \item[\(\rm (ii)\)] \(h\) satisfies almost everywhere on \([t_0, +\infty]\) that there exists \(i \in \{1,\cdots,m\}\) such that
        \[
        h(t) = h_i(t) ,\qquad \frac{d}{dt} h(t) = \frac{d}{dt} h_i(t).
        \]
    \end{itemize}
\end{lem}
\begin{proof}
See \cite[Lemma 4.12]{sonntag2024fastgradientflow}. \qed 
\end{proof} 

Next we provide the proof of \cref{lem:LyapunovInequal}:

\noindent{\it Proof of \cref{lem:LyapunovInequal}\ } Using \cref{lem:diffequaldiffi}, we obtain that $\mathcal E_z(t)$ is differentiable for almost all $t\ge t_0$, and  there exists $i_0$ such that
$$
\mathcal E_z(t)=\mathcal E_z^{i_0}(t),\qquad\frac{d}{dt}\mathcal E_z(t)=\frac{d}{dt}\mathcal E_z^{i_0}(t){},  
$$
for almost all $t\ge t_0$. Thus, there exists $i_0$ such that  
$$
\begin{aligned}  
\frac{d}{dt}\mathcal E_z(t)&=\frac{d}{dt}\mathcal E_z^{i_0}(t)\\  
&\le \frac{3-\beta}{2(\alpha -1)}t\|\dot x(t)\|^2-\frac{t}{\alpha -1}\Big(f_{i_0}(x(t))-f_i(z)\Big)\\  
&\qquad +\frac{t}{\alpha -1}\Big(f_{i_0}(x(t))-f_{i_0}(z)\Big)+\frac{\alpha -\beta}{(\alpha -1)t^{p-1}}\|\dot x(t)\|\|x(t)-z\|\\  
&=\frac{3-\beta}{2(\alpha -1)}t\|\dot x(t)\|^2+\frac{\alpha -\beta}{(\alpha -1)t^{p-1}}\|\dot x(t)\|\|x(t)-z\|.   
\end{aligned}  
$$
The proof is complete.\qed 

\section{Supplement to \cref{sec:Algorithm}}\label{appendix:algorithm}
\subsection{Discretization of \cref{eq:MAVNG}}\label{appendix:subsection-algorithm}
\begin{equation}\label{eq:Discret-MAVNG3}
\left\{\begin{aligned}
\pi_k &=(x_k-x_{k-1})-\frac{\alpha -3}{k -1}\frac{\|x_k-x_{k-1}\|}{\|\proj_{C(x_k)}(0)\|}\proj _{C(x_k)}(0),\\
y_k&=x_k+\frac{k-1}{k+\alpha -1}\pi_k,\\
x_{k+1}&=y_k-s\frac{k-1}{k+\alpha -1}\proj_{C(y_k)}\left(\frac{1}{s}\pi_k\right).
\end{aligned}\right.
\end{equation}

Indeed, it is straightforward to observe that \cref{eq:Discret-MAVNG3} imply, 
\begin{equation}\label{eq:Discret-MAVNG2}
\begin{aligned}
x_{k+1}-x_k&=\frac{k-1}{k+\alpha -1}\left(\pi_k-s\proj_{C(y_k)}\left(\frac{1}{s}\pi_k\right)\right)\\
&=\frac{k-1}{k+\alpha -1}\left(\pi_k -\proj_{sC(y_k)}(\pi_k)\right)\\
&=-\frac{k-1}{k+\alpha -1}\proj_{sC(y_k)-\pi_k}(0)\\
&=-\frac{1}{1+\frac{\alpha}{k-1}}\proj_{sC(y_k)-\pi_k}(0).
\end{aligned}
\end{equation}
Following the \cref{lem:projection_equation}, we obtain:
$$
\begin{aligned}
-\frac{\alpha }{k-1}(x_{k+1}-x_k)=\proj_{sC(y_k)+\frac{\alpha -3}{k-1}\frac{\|x_k-x_{k-1}\|}{\|\proj_{C(x_k)}(0)\|}\proj_{C(x_k)}(0)+x_{k+1}-2x_k+x_{k-1}}(0).
\end{aligned}
$$
Similarly, let $t \approx t_k=(k-1)\sqrt{s}$, yielding:
\[
\begin{aligned}
-\frac{\alpha\sqrt{s}}{t}(\dot{x}(t)\sqrt{s}+o(\sqrt{s}))&=\proj_{sC(x(t))+\frac{(\alpha -3)\sqrt{s}}{t}\frac{\|\dot{x}(t)\|\sqrt{s}+o(\sqrt{s})}{\|\proj_{C(x(t))}(0)\|}\proj_{C(x(t))}(0)+\ddot{x}(t)+o(s)}(0)\\
&=s\proj_{C(x(t))+\frac{(\alpha -3)}{t}\frac{\|\dot{x}(t)\|}{\|\proj_{C(x(t))}(0)\|}\proj_{C(x(t))}(0)+\ddot{x}(t)}(0)+o(s).
\end{aligned}
\]
Comparing the coefficients of $s$, we derive:
\[
\frac{\alpha}{t}\dot{x}(t)+\proj_{C(x(t))+\frac{\alpha -3}{t}\frac{\|\dot{x}(t)\|}{\|\proj_{C(x(t))}(0)\|}\proj_{C(x(t))}(0)+\ddot{x}(t)}(0)=0.
\]
This implies that \cref{eq:MAVNG} is the continuous-time form of the iterative scheme \cref{eq:Discret-MAVNG3}.
\subsection{proofs of \cref{lem:boundedsequence,eq:Lyapunovk-Lyapunov1}}\label{appendix:subsection-Convergence}
 We give some notations as follows, 
\begin{itemize}
    \item[$\bullet$]  For any $z\in \R^n$, let  
    \begin{equation*} 
    \sigma_k(z):=\min_{i=1,\cdots,m}\Big(f_i(x_k)-f_i(z)\Big).
    \end{equation*}
    \item[$\bullet$] Let
    \begin{equation*}
    \Delta x_k=x_k-x_{k-1},\quad \phi_k = 2k+\alpha -3.
    \end{equation*}
    \item[$\bullet$] Let \(r_k\) be represented by the following vector.
    \begin{equation*}
    r_k=\frac{\|\Delta x_k\|}{\|\proj_{C(x_k)}(0)\|}\proj_{C(x_k)}(0).
    \end{equation*}
    \item[$\bullet$] Let
    \begin{equation*}
    \xi_k =\frac{k+\alpha -2}{\alpha -1},\quad \nu_k =\frac{2(k+\alpha -2)(k+\alpha -4)}{\alpha -1}.
    \end{equation*}
    \item[$\bullet$] Let 
    $$
    R = \sup_{x \in \mathcal{L}(F, F(x_0))} \|x\|  .
    $$
    
\end{itemize}
We have the following lemma:  

\begin{lem}\label{lem:Inequaityofsigma} 
    The following inequalities hold:  
    $$
    \begin{aligned}
    \sigma _k(z)&\le -\frac{1}{s}\left\langle x_k-y_{k-1},y_{k-1}-z \right\rangle -\frac{1}{2s}\|x_k-y_{k-1}\|^2, \\
    \sigma _{k}(z)-\sigma _{k-1}(z)&\le \max_{i=1,\cdots,m}(f_i(x_k)-f_i(x_{k-1}))\\&\le -\frac{1}{s}\left\langle x_k-y_{k-1},y_{k-1}-x_{k-1} \right\rangle -\frac{1}{2s}\|x_k-y_{k-1}\|^2.
    \end{aligned}
    $$  
\end{lem}
\begin{proof}
    The proof follows the same arguments as in \cite[Lemma 6.3, Lemma 6.5]{sonntag2024fastNestrovAlgorithm} and is omitted here.\qed 
\end{proof} 

\begin{coro}\label{coro:sigammore}
    For $k_1< k_2$, we have  
    $$
    \sigma_{k_2}(z)-\sigma _{k_1}(z)\le \frac{1}{2s}\Big[\|\Delta x_{k_1}\|^2-\|\Delta {x_{k_2}}\|^2\Big]+\frac1{2s}\sum_{k=k_1+1}^{k_2}Q_k\|\Delta x_{k-1}\|^2.
    $$  
    where $Q_k =\left(\frac{k+\alpha -5}{k+\alpha -2}\right)^2-1$. 
\end{coro} 
\begin{proof}
    For convenience of presentation, we first set the following parameters:
    $$
    \beta_k =\frac{\alpha }{k+\alpha -1},\quad \gamma_k =\frac{\alpha -3}{k+\alpha -1}.
    $$  
    Then, according to the definition of $y_k$, we have
    $$
    y_{k-1}=x_{k-1}+(1-\beta_{k-1})(x_{k-1}-x_{k-2})-\gamma_{k-1}r_{k-1}.
    $$  
    Based on this, we provide an estimate for $\|y_{k-1}-x_{k-1}\|$:
    $$
    \begin{aligned}
    \|y_{k-1}-x_{k-1}\|^2&=(1-\beta_{k-1})^2\|\Delta x_{k-1}\|^2+\gamma_{k-1}^2\|\Delta x_{k-1}\|^2\\&\qquad  -2(1-\beta_{k-1})\gamma_{k-1}\left\langle \Delta x_{k-1},r_{k-1} \right\rangle \\
    &\le \Big((1-\beta_{k-1})+\gamma_{k-1}\Big)^2\|\Delta x_{k-1}\|^2\\&=\left(\frac{k+\alpha -5}{k+\alpha -2}\right)^2\|\Delta x_{k-1}\|^2 .
    \end{aligned}
    $$  
    Furthermore, by \cref{lem:Inequaityofsigma} , we obtain
    \begin{equation}\label{eq:sigmadiff}
    \begin{aligned}
    \sigma_k(z)-\sigma_{k-1}(z)&\le \max_{i=1,\cdots,m}(f_i(x_k)-f_i(x_{k-1}))
    \\&\le -\frac{1}{2s}\left\langle x_{k}-y_{k-1},y_{k-1}-x_{k-1} \right\rangle -\frac{1}{2s}\|x_{k}-y_{k-1}\|^2\\
    &=\frac{1}{2s}\Big[\|y_{k-1}-x_{k-1}\|^2-\|x_k-x_{k-1}\|^2\Big]\\
    &\le \frac{1}{2s}\left[\left(\frac{k+\alpha -5}{k+\alpha -2}\right)^2\|\Delta x_{k-1}\|^2-\|\Delta x_k\|^2\right]\\
    &=\frac1{2s}\left(\left(\frac{k+\alpha -5}{k+\alpha -2}\right)^2-1\right)\|\Delta x_{k-1}\|^2\\&\qquad +\frac{1}{2s}\Big[\|\Delta x_{k-1}\|^2-\|\Delta x_k\|^2\Big].
    \end{aligned}
    \end{equation}
    Summing over $k$ from $k_1+1$ to $k_2$ completes the proof. \qed 
\end{proof}
{\it Proof of \cref{lem:boundedsequence}}
    Using \cref{lem:Inequaityofsigma}  and a reasoning similar to equation \cref{eq:sigmadiff}, we obtain that for any \( p \ge 1 \),  
\begin{equation}
\begin{aligned} 
    f_i(x_p) - f_i(x_{p-1}) &\max_{i=1,\cdots,m}(f_i(x_k)-f_i(x_{k-1}))\\&\le \frac{1}{2s}\Big[\|\Delta x_{p-1}\|^2 - \|\Delta x_p\|^2\Big]+\frac1{2s}\left(\left(\frac{k+\alpha -5}{k+\alpha-2 }\right)^2-1\right)\\&\le \frac{1}{2s}\Big[\|\Delta x_{p-1}\|^2 - \|\Delta x_p\|^2\Big]
\end{aligned}
\end{equation} 
    Summing from \( p = 2 \) to \( p = k \) and noting that \( x_1 = x_0 \), the conclusion follows.\qed 

\begin{lem} \label{lem:mathcalG}
    Let
    \begin{equation}
    \mathcal G_z(k)=\frac{2(k+\alpha -2)^2s}{\alpha -1}\sigma_k(z).
    \end{equation}
    Then
    \begin{equation}
    \begin{aligned}
    &\mathcal G_z(k)-\mathcal G_z(k-1)\\&\le(\alpha -1)\xi_k \bigg\langle y_{k-1}-x_k,2\xi_k y_{k-1}-\frac{2(k-1)}{\alpha -1}x_{k-1}-2z-\xi_k (y_{k-1}-x_k)\bigg\rangle\\
    &\quad +2(\alpha - 3)\xi_k \Big\langle y_{k-1}-x_k,z-x_{k-1}\Big\rangle-\frac{2s}{\alpha -1}\sigma_{k-1}(z).
    \end{aligned}
    \end{equation}
\end{lem}
\begin{proof} 
    Based on \cref{lem:Inequaityofsigma} , we have
    $$
    \begin{aligned}
    &\mathcal G_z(k)-\mathcal G_z(k-1)+\frac{2s}{\alpha -1}\sigma _{k-1}(z)\\
    =\ &\frac{4(k+\alpha -2)s}{\alpha -1}\sigma_k(z)+\frac{2(k+\alpha -2)(k+\alpha -4)s}{\alpha -1}\Big(\sigma_k(z)-\sigma _{k-1}(z)\Big)\\
    \le \ & -\frac{4(k+\alpha -2)s}{\alpha -1}\bigg(\frac1s\Big\langle x_k-y_{k-1},y_{k-1}-z\Big\rangle+\frac1{2s}\Big\|x_k-y_{k-1}\Big\|^2\bigg)\\
    &-\frac{2(k+\alpha -2)(k+\alpha -4)s}{\alpha -1}\bigg(\frac1{s}\Big\langle x_k-y_{k-1},y_{k-1}-x_{k-1}\Big\rangle+\frac1{2s}\Big\|x_k-y_{k-1}\Big\|^2\bigg).
    \end{aligned}
    $$
    Note that $4\xi_k +\nu_k =2(\alpha -1)\xi_k ^2$ and hence
    $$
    \begin{aligned}
    &\mathcal G_z(k)-\mathcal G_z(k-1)+\frac{2s}{\alpha -1}\sigma _{k-1}(z)\\
    \le \ &  4\xi_k \bigg(\Big\langle y_{k-1}-x_k,y_{k-1}-z\Big\rangle-\frac12\|x_k-y_{k-1}\|^2\bigg)\\& +\nu_k\bigg(\Big\langle y_{k-1}-x_{k},y_{k-1}-x_{k-1}\Big\rangle-\frac1{2}\Big\|x_k-y_{k-1}\Big\|^2\bigg)\\
    =\ &(4\xi_k +\nu_k )\bigg(\Big\langle y_{k-1}-x_k,y_{k-1}\Big\rangle-\frac12\|x_k-y_{k-1}\|^2\bigg)-4\xi_k\Big\langle y_{k-1}-x_k,z\Big\rangle\\&\quad -\nu_k \Big\langle y_{k-1}-x_k,x_{k-1}\Big\rangle\\
    =\ &2(\alpha -1)\xi_k^2\bigg(\Big\langle y_{k-1}-x_k,y_{k-1}\Big\rangle-\frac12\|x_k-y_{k-1}\|^2\bigg)-\Big\langle y_{k-1}-x_k,4\xi_k z+\nu_k x_{k-1}\Big\rangle\\
    =\ &2(\alpha -1)\xi_k^2\left(\Big<y_{k-1}-x_k,y_{k-1}\Big>-\frac12\|x_k-y_{k-1}\|^2\right)\\&\quad -(\alpha -1)\xi_k\left<y_{k-1}-x_k,\frac{4 }{\alpha -1}z{+\frac{\nu_k }{(\alpha -1)\xi_k}x_{k-1}}\right>\\
    =\ &2(\alpha -1)\xi_k^2\left(\Big<y_{k-1}-x_k,y_{k-1}\Big>-\frac12\|x_k-y_{k-1}\|^2\right)\\&\quad -(\alpha -1)\xi_k\left<y_{k-1}-x_k,\frac{4-2(\alpha -1) }{\alpha -1}z+2z{+\frac{\nu_k }{(\alpha -1)\xi_k}x_{k-1}}\right>\\
    =\ &(\alpha -1)\xi_k \bigg\langle y_{k-1}-x_k,2\xi_k y_{k-1}-\frac{2(k-1)}{\alpha -1}x_{k-1}-2z-\xi_k (y_{k-1}-x_k)\bigg\rangle\\
    &\quad +2(\alpha - 3)\xi_k \Big\langle y_{k-1}-x_k,z-x_{k-1}\Big\rangle.
    \end{aligned}
    $$
    Thus, the proof is completed.\qed 
\end{proof}
\begin{lem}\label{lem:Lyapunovestimate}
    Let
    \begin{equation}
    \mathcal E_z(k):=\mathcal G_z(k)+2\left\|x_k-z+\frac{k-1}{\alpha -1}\Delta x_k\right\|^2+\frac{(\alpha -3)(k-1)^2}{(\alpha -1)^2}\|\Delta x_k\|^2.
    \end{equation}
    Then
    \begin{equation}
    \begin{aligned} 
    \mathcal E_z(k)-\mathcal E_z(k-1)&\le -\frac{\alpha -3}{\alpha -1}\phi_k \|\Delta x_k\|^2+\frac{\alpha -3}{\alpha -1}(\phi_{k-1}-2)\|\Delta x_{k-1}\|^2\\&\quad +4\frac{\alpha -3}{\alpha -1}\| \Delta x_{k-1}\|\|x_{k-1}-z\|-\frac{2s}{\alpha -1}\sigma_{k-1}(z).
    \end{aligned}
    \end{equation}
\end{lem}
\begin{proof}
    Let
    \begin{equation}\label{eq:Definition_of_z_k} 
    z_k =x_k +\frac{k-1}{\alpha -1}\Delta x_k ,\quad w_k = z_k+z_{k-1}-x_k-x_{k-1}.
    \end{equation}
    Then
    \begin{equation}\label{eq:Definition_of_z_k-1}
\begin{aligned}
    z_{k-1}&= x_{k-1}+\frac{k-2}{\alpha -1}\Delta x_{k-1}\\
    & = \frac{k+\alpha -2-(k-1)}{\alpha -1} x_{k-1}+\frac{k-2}{\alpha -1}\Delta x_{k-1}\\
    &=\frac{k+\alpha -2}{\alpha -1}x_{k-1}+\frac{k-2}{\alpha -1}\Delta x_{k-1}-\frac{k-1}{\alpha -1}x_{k-1}\\
    &=\frac{k+\alpha -2}{\alpha -1}\left(x_{k-1}+\frac{k-2}{k+\alpha -2}\Delta x_{k-1}\right)-\frac{k-1}{\alpha -1}x_{k-1}\\
    &= \frac{k+\alpha -2}{\alpha -1}\left(y_{k-1}+x_{k-1}-y_{k-1}+\frac{k-2}{k+\alpha -2}\Delta x_{k-1}\right)-\frac{k-1}{\alpha -1}x_{k-1}\\
    &\xlongequal{\text{Definitin of $y_{k-1}$}}\frac{k+\alpha -2}{\alpha -1}\bigg[y_{k-1}+x_{k-1}\\&\qquad -\left(x_{k-1}+\frac{k-2}{k+\alpha -2}\Delta x_{k-1}-\frac{\alpha -3}{k+\alpha -2}r_{k-1}\right)+\frac{k-2}{k+\alpha -2}\Delta x_{k-1}\bigg] \\&\qquad-\frac{k-1}{\alpha -1}x_{k-1}\\
    & = \frac{k+\alpha -2}{\alpha -1}\left(y_{k-1}+\frac{\alpha -3}{k+\alpha -2}r_{k-1}\right)-\frac{k-1}{\alpha -1}x_{k-1}\\
    &=\xi_k y_{k-1}+\frac{\alpha -3}{\alpha -1}r_{k-1}-\frac{k-1}{\alpha -1}x_{k-1}. 
\end{aligned}
    \end{equation}
    and
    \begin{equation} \label{eq:Deltax_k-1-r_k-1}
    \begin{aligned}
    \frac{k-2}{\alpha -1}\Delta x_{k-1}-\frac{\alpha -3}{\alpha -1}r_{k-1}&=\frac{k+\alpha -2}{\alpha -1}\left(\frac{k-2}{k+\alpha -2}\Delta x_{k-1}-\frac{\alpha -3}{k+\alpha -2}r_{k-1}\right)\\
    &=\frac{k+\alpha -2}{\alpha -1}\bigg(\frac{k-2}{k+\alpha -2}\Delta x_{k-1}\\&\qquad +\Big[y_{k-1}-x_{k-1}-\frac{k-2}{k+\alpha -2}\Delta x_{k-1}\Big]\bigg)\\
    &=\frac{k+\alpha -2}{\alpha -1}\left(y_{k-1}-x_{k-1}\right)\\
    &=\frac{k+\alpha -2}{\alpha -1}\left(\Delta x_k+(y_{k-1}-x_k)\right)\\
    &=\xi_k \Big(\Delta x_k +(y_{k-1}-x_k)\Big).
    \end{aligned}
    \end{equation}
    Thus,
    \begin{equation}\label{eq:z_k-z_k-1}
\begin{aligned}
z_k-z_{k-1}&=\Delta x_k+\frac{k-1}{\alpha -1}\Delta x_k -\frac{k-2}{\alpha -1}\Delta x_{k-1}\\
&=\frac{k+\alpha -2}{\alpha -1}\Delta x_k-\frac{k-2}{\alpha -1}\Delta x_{k-1}\\
&\xlongequal{\cref{eq:Deltax_k-1-r_k-1}}\frac{k-2}{\alpha -1}\Delta x_{k-1}-\frac{k+\alpha -2}{\alpha -1}(y_{k-1}-x_k)-\frac{\alpha -3}{\alpha -1}r_{k-1}-\frac{k-2}{\alpha -1}\Delta x_{k-1}\\
&=-\xi_k (y_{k-1}-x_k)-\frac{\alpha -3}{\alpha -1}r_{k-1},\\
\end{aligned}
    \end{equation}
    \begin{equation}\label{eq:z_k+z_k-1}
    \begin{aligned}
    z_k+z_{k-1}&=z_k-z_{k-1}+2z_{k-1}\\
    &\xlongequal{\cref{eq:Definition_of_z_k,eq:z_k-z_k-1}}-\xi_k (y_{k-1}-x_k)-\frac{\alpha -3}{\alpha -1}r_{k-1}\\&\qquad +2\xi_k y_{k-1}+2\frac{\alpha -3}{\alpha -1}r_{k-1}-2\frac{k-1}{\alpha -1}x_{k-1}\\
    &=-\xi_k(y_{k-1}-x_k)+2\xi_k y_{k-1}-2\frac{k-1}{\alpha -1}x_{k-1}+\frac{\alpha -3}{\alpha -1}r_{k-1}.
    \end{aligned}
    \end{equation}
    Using \cref{eq:Definition_of_z_k} and \cref{eq:Definition_of_z_k-1}, and noting
    \begin{equation}\label{eq:lem5.4.5}
    \frac{k -1}{\alpha -1}+\xi_k =\frac{2k+\alpha -3}{\alpha -1}=\frac{\phi_k}{\alpha -1}.
    \end{equation}
    we have
    \begin{equation}\label{eq:computing_w_k}
\begin{aligned}
w_k &= z_k+z_{k-1}-x_k -x_{k-1}\\
&\xlongequal{\cref{eq:Definition_of_z_k}}x_k+\frac{k-1}{\alpha -1}\Delta x_{k}+x_{k-1}+\frac{k-2}{\alpha -1}\Delta x_{k-1}-x_k-x_{k-1}\\
&=\frac{k-1}{\alpha -1}\Delta x_k +\frac{k-2}{\alpha -1}\Delta x_{k-1}\\
&\xlongequal{\cref{eq:Deltax_k-1-r_k-1}}\frac{k-1}{\alpha -1}\Delta x_k +\frac{\alpha -3}{\alpha -1}r_{k-1}+\xi_k \Big(\Delta x_k+(y_{k-1}-x_k)\Big)\\
&=\frac{2k+\alpha -3}{\alpha -1}\Delta x_k+\frac{\alpha -3}{\alpha -1}r_{k-1}+\xi_k \Big(y_{k-1}-x_k\Big)\\
&=\frac{\phi_k}{\alpha -1}\Delta x_k+\frac{\alpha -3}{\alpha -1}r_{k-1}+\xi_k \Big(y_{k-1}-x_k\Big)
\end{aligned}
    \end{equation}
    We now analyze the difference:
    \begin{equation}\label{eq:lem5.4.7}
\begin{aligned}
&2\left\|x_k-z+\frac{k-1}{\alpha -1}\Delta x_k\right\|^2-2\left\|x_{k-1}-z+\frac{k-2}{\alpha -1}\Delta x_{k-1}\right\|^2\\
=\ &2\|z_k-z\|^2-2\|z_{k-1}-z\|^2\\
=\ &2\Big<z_k-z_{k-1},z_k+z_{k-1}-2z\Big>\\
\xlongequal{\cref{eq:z_k+z_k-1}}\ &-2\xi_k \Big<y_{k-1}-x_k,z_{k}+z_{k-1}-2z\Big>-2\frac{\alpha -3}{\alpha -2}\Big<r_{k-1},z_k+z_{k-1}-2z\Big>\\
\xlongequal{\cref{eq:Definition_of_z_k}}\ &-2\xi_k \Big<y_{k-1}-x_k,z_{k}+z_{k-1}-2z\Big>-2\frac{\alpha -3}{\alpha -2}\Big<r_{k-1},x_k+x_{k-1}-2z\Big>\\
&\qquad -2\frac{\alpha -3}{\alpha -1}\Big<r_{k-1},w_k\Big>
\end{aligned}
    \end{equation}
    and
    \begin{equation}\label{eq:lem5.4.8}
\begin{aligned}
&\frac{(\alpha -3)(k-1)^2}{(\alpha -1)^2}\|\Delta x_k\|^2-\frac{(\alpha -3)(k-2)^2}{(\alpha -1)^2}\|\Delta x_{k-1}\|^2\\
\xlongequal{\cref{eq:Definition_of_z_k}}\ & (\alpha -3)\Big(\|z_k-x_k\|^2-\|z_{k-1}-x_{k-1}\|^2\Big)\\
=\ &(\alpha -3)\Big<z_k-z_{k-1}-\Delta x_k,w_k\Big>\\
\xlongequal{\cref{eq:z_k-z_k-1}}\ &(\alpha -3)\Big<-\frac{\alpha -3}{\alpha -1}r_{k-1}-\xi_k\Big(y_{k-1}-x_k\Big)-\Delta x_k,w_k\Big>\\
=\ &-(\alpha -3)\Big<\frac{\alpha -3}{\alpha -1}r_{k-1}+\Delta x_k,w_k\Big>-(\alpha -3)\xi_k\Big<y_{k-1}-x_k,w_k\Big>
\end{aligned}
    \end{equation}
    Thus,
    \begin{equation}\label{eq:lem5.4.9}
\begin{aligned}
&\mathcal E_z(k)-\mathcal E_z(k-1)\\
=\ &\underbrace{-(\alpha -3)\Big\langle  r_{k-1}+\Delta x_k,w_k\Big\rangle-2\frac{\alpha -3}{\alpha -1}\Big\langle  r_{k-1},x_k+x_{k-1}-2z\Big\rangle}_{L_{z,1}(k)}\\
&\quad \underbrace{-2\xi_k \Big\langle y_{k-1}-x_k,z_k+z_{k-1}-2z\Big\rangle-(\alpha -3)\xi_k \Big\langle y_{k-1}-x_k,w_k\Big\rangle}_{L_{z,2}(k)}\\
&\quad +\ \mathcal G_z(k)-\mathcal G_z(k-1).
\end{aligned}
    \end{equation}
    For each part:
    \begin{equation}\label{eq:Definition_of_z_k0}
    \begin{aligned}
    &L_{1,z}(k)+4\frac{\alpha -3}{\alpha -1}\Big<r_{k-1},x_{k-1}-z\Big>\\
    =\ &-(\alpha -3)\Big<r_{k-1}+\Delta x_{k},w_k\Big>-2\frac{\alpha -3}{\alpha -1}\Big<r_{k-1},\Delta x_k\Big>\\
    \xlongequal{\cref{eq:computing_w_k}}\ &-(\alpha -3)\left< r_{k-1}+\Delta x_k,\frac{\phi_k}{\alpha -1}\Delta x_k+\frac{\alpha -3}{\alpha -1}r_{k-1}+\xi_k (y_{k-1}-x_k)\right>\\
    &-2\frac{\alpha -3}{\alpha -1}\Big<r_{k-1},\Delta x_k\Big>\\
    =\ &-(\alpha -3)\Big<r_{k-1}+\Delta x_k,\xi_k (y_{k-1}-x_k)\Big>\\
    &-(\alpha -3)\left<r_{k-1}+\Delta x_k,\frac{\phi_k}{\alpha - 1}\Delta x_k+\frac{\alpha -3}{\alpha -1}r_{k-1}\right>\\
    &-2\frac{\alpha -3}{\alpha -1}\Big<r_{k-1},\Delta x_k\Big>\\
    =\ &-(\alpha -3)\Big<r_{k-1} ,\xi_k (y_{k-1}-x_k)\Big>-(\alpha -3)\Big<\Delta x_k,\xi_k (y_{k-1}-x_k)\Big>\\
    &-(\alpha -3)\left<r_{k-1},\frac{\phi_k}{\alpha -1}\Delta x_k\right>-(\alpha -3)\frac{\alpha -3}{\alpha -1}\|r_{k-1}\|^2\\
    &-(\alpha -3)\frac{\phi_k}{\alpha -1}\|\Delta x_k\|^2-(\alpha -3)\left<\Delta x_k,\frac{\alpha -3}{\alpha -1}r_{k-1}\right>\\
    &-2\frac{\alpha -3}{\alpha -1}\Big<r_{k-1},\Delta x_k\Big>\\
    =\ &-(\alpha -3)\Big<r_{k-1} ,\xi_k (y_{k-1}-x_k)\Big>-(\alpha -3)\Big<\Delta x_k,\xi_k (y_{k-1}-x_k)\Big>\\
    &-(\alpha -3)\frac{\alpha -3}{\alpha -1}\|r_{k-1}\|^2-(\alpha -3)\frac{\phi_k}{\alpha -1}\|\Delta x_k\|^2\\
    &-\frac{\alpha -3}{\alpha -1}(2+\phi_k+(\alpha-3))\Big<r_{k-1},\Delta x_k\Big>\\
    =\ &-\frac{\alpha -3}{\alpha -1}\Big(\phi_k\|\Delta x_k\|^2+(\alpha -3)\|r_{k-1}\|^2\Big)-2(\alpha -3)\xi_k\Big<r_{k-1},\Delta x_k\Big>\\
    &-(\alpha -3)\xi_k\Big<r_{k-1}, y_{k-1}-x_k\Big>-(\alpha -3)\xi_k\Big<\Delta x_k,y_{k-1}-x_k\Big>\\
    =\ &\underbrace{-\frac{\alpha -3}{\alpha -1}\Big(\phi_k\|\Delta x_k\|^2+(\alpha -3)\|r_{k-1}\|^2+2(\alpha -1)\xi_k\Big<r_{k-1},\Delta x_k+(y_{k-1}-x_k)\Big>\Big)}_{\widetilde {L_{z,1}}(k)}\\
    &+\underbrace{(\alpha -3)\xi_k\Big<r_{k-1}, y_{k-1}-x_k\Big>-(\alpha -3)\xi_k\Big<\Delta x_k,y_{k-1}-x_k\Big>}_{\overline{L_{z,1}}(k)}\\
    \end{aligned}
    \end{equation}
    where
    \begin{equation}\label{eq:barL1} 
    \begin{aligned}
    \overline{L_{z,1}}(k)&=(\alpha -3)\xi_k \Big\langle r_{k-1},y_{k-1}-x_k\Big\rangle-(\alpha -3)\xi_k \Big\langle\Delta x_k ,y_{k-1}-x_k\Big\rangle\\
    &=(\alpha -3)\xi_k \Big\langle y_{k-1}-x_k, r_{k-1}-\Delta x_k\Big\rangle.
    \end{aligned}
    \end{equation}
    \begin{equation}\label{eq:widetilde_L}
    \begin{aligned}
    \widetilde {L_{1,z}}(k)&\xlongequal{\xi_k}-\frac{\alpha -3}{\alpha -1}\Bigg(\phi_k\|\Delta x_k\|^2+(\alpha -3)\|r_{k-1}\|^2\\
    &\qquad +2\bigg<r_{k-1},(k+\alpha -2)\Delta x_k\\&\qquad \qquad +(k+\alpha -2)\left(x_{k-1}+\frac{k-2}{k+\alpha -2}\Delta x_{k-1}-\frac{\alpha -3}{k+\alpha -2}r_{k-1}-x_k\right)\bigg>\Bigg)\\
    &=-\frac{\alpha -3}{\alpha -1}\Bigg(\phi_k\|\Delta x_k\|^2+(\alpha -3)\|r_{k-1}\|^2
    \\&\qquad +2\bigg<r_{k-1},(k+\alpha -2)\Delta x_k\\&\qquad \qquad +\Big(-(k+\alpha -2)\Delta x_k+({k-2})\Delta x_{k-1}-({\alpha -3})r_{k-1}\Big)\bigg>\Bigg)\\
    &=-\frac{\alpha -3}{\alpha -1}\Bigg(\phi_k\|\Delta x_k\|^2+(\alpha -3)\|r_{k-1}\|^2\\
    &\qquad  +2\bigg<r_{k-1},({k-2})\Delta x_{k-1}-({\alpha -3})r_{k-1}\bigg>\Bigg)\\
    &=-\frac{\alpha -3}{\alpha -1}\bigg(\phi_k \|\Delta x_k\|^2-(\alpha -3)\|r_{k-1}\|^2+2(k-2)\Big<r_{k-1},\Delta x_{k-1}\Big>\bigg)\\
    &\le -\frac{\alpha -3}{\alpha -1}\bigg(\phi_k\|\Delta x_k\|^2-(\alpha -3)\|r_{k-1}\|^2-2(k-2)\|r_{k-1}\|\|\Delta x_{k-1}\|\bigg)\\
    &\xlongequal{\|r_{k-1}\|=\|\Delta x_{k-1}\|} -\frac{\alpha -3}{\alpha -1}\bigg(\phi_k\|\Delta x_k\|^2-(2k-\alpha -7)\|\Delta x_{k-1}\|^2\bigg)\\
    &\xlongequal{\text{Definition of }\phi_k} -\frac{\alpha -3}{\alpha -1}\bigg(\phi_k\|\Delta x_k\|^2-(\phi_{k-1}-2)\|\Delta x_{k-1}\|^2\bigg)
    \end{aligned}
    \end{equation}
    Combine with \cref{eq:Definition_of_z_k0,eq:barL1,eq:widetilde_L},
  \begin{equation}\label{eq:Definitionofzk2}
   \begin{aligned}
   L_{1,z}(k)&= \widetilde{L_{1,z}}(k)+ \overline{L_{1,z}}(k)-4\frac{\alpha -3}{\alpha -1}\Big<r_{k-1},x_{k-1}-z\Big>\\
   &\le -\frac{\alpha -3}{\alpha -1}\bigg(\phi_k\|\Delta x_k\|^2-(\phi_{k-1}-2)\|\Delta x_{k-1}\|^2\bigg)\\
   &\qquad +\ 4\frac{\alpha -3}{\alpha -1}\|\Delta x_{k-1}\|\|x_{k-1}-z\|+\overline{L_{1,z}}(k)
   \end{aligned}
    \end{equation}
    According to the definition of \(w_k\),
    \begin{equation}\label{eq:Definitionofzk3}
    \begin{aligned}
    &2(z_k+z_{k-1}-2z)+(\alpha -3)w_k\\
    =\ &2(z_k+z_{k-1}-2z)+(\alpha -3)\big((z_k+z_{k-1}-2z)-(x_k+x_{k-1}-2z)\big)\\
    =\ &(\alpha -1)\bigg(-\xi_k (y_{k-1}-x_k)+2\xi_k y_{k-1}-\frac{2(k-1)}{r-1}x_{k-1}+\frac{\alpha -3}{\alpha -1} r_{k-1}-2z\bigg)\\
    & -2(\alpha -3)(x_{k-1}-z)-(\alpha -3)\Delta x_k\\
    =\ &(\alpha -1)\bigg(-\xi_k (y_{k-1}-x_k)+2\xi_k y_{k-1}-\frac{2(k-1)}{r-1}x_{k-1}-2z\bigg)\\
    & -(\alpha -3)(x_{k-1}-z)-(\alpha -3)(\Delta x_k- r_{k-1}).
    \end{aligned}
    \end{equation}
    Therefore, by \cref{lem:mathcalG} and \cref{eq:barL1}
    \begin{equation}\label{eq:Definition_of_z_k4}
\begin{aligned}
L_{2,z}(k)&=-2\Big\langle y_{k-1}-x_k,2(z_k+z_{k-1}-2z)+(\alpha -3)w_k\Big\rangle\\
&=-(\alpha -1)\xi_k \Big\langle y_{k-1}-x_k,-\xi_k(y_{k-1}-x_k)+2\xi_k y_{k-1}-\frac{2(k-1)}{\alpha -1}x_{k-1}-2z\Big\rangle\\
&\quad +2(\alpha -3)\xi_k \Big\langle y_{k-1}-x_k,x_{k-1}-z\Big\rangle+(\alpha -3)\xi_k \Big\langle y_{k-1}-x_k,\Delta x_{k}-r_{k-1}\Big\rangle\\
&\le \mathcal G_z(k-1)-\mathcal G_z(k)-\frac{2s}{\alpha -1}\sigma_{k-1}(z)-\overline{L_{1,z}}(k)
\end{aligned}
    \end{equation}
    Combining \cref{eq:lem5.4.9}, \cref{eq:Definitionofzk2}, and \cref{eq:Definition_of_z_k4}, we have
\begin{equation}
\begin{aligned}
\mathcal E_z(k)-\mathcal E_z(k-1) &= L_{1,z}(k)+L_{2,z}(k)+ \mathcal G_z(k)-\mathcal G_z(k-1)\\
&\le -\frac{\alpha -3}{\alpha -1}\bigg(\phi_k\|\Delta x_k\|^2-(\phi_{k-1}-2)\|\Delta x_{k-1}\|^2\bigg)\\
&\qquad +\ 4\frac{\alpha -3}{\alpha -1}\|\Delta x_{k-1}\|\|x_{k-1}-z\|+\overline{L_{1,z}}(k)\\
&\qquad +\ \mathcal G_z(k-1)-\mathcal G_z(k)-\frac{2s}{\alpha -1}\sigma_{k-1}(z)-\overline{L_{1,z}}(k)\\
&\qquad +\  \mathcal G_z(k)-\mathcal G_z(k-1)\\
&\le -\frac{\alpha -3}{\alpha -1}\bigg(\phi_k\|\Delta x_k\|^2-(\phi_{k-1}-2)\|\Delta x_{k-1}\|^2\bigg)\\
&\qquad +\ 4\frac{\alpha -3}{\alpha -1}\|\Delta x_{k-1}\|\|x_{k-1}-z\|
-\frac{2s}{\alpha -1}\sigma_{k-1}(z)\\
\end{aligned}
\end{equation}\qed 
\end{proof}

{\it Proof of \cref{eq:Lyapunovk-Lyapunov1}}
    	For any \(k > 1\) and \(1< j \le k\), according to the \cref{coro:sigammore}, we know  
    \begin{equation}
    -\sigma_{j-1}(z)\le -\sigma _{k}(z)+\frac1{2s}\Big[\|\Delta x_{j-1}\|^2-\|\Delta x_{k}\|^2\Big]+\frac1{2s}\sum_{l=j}^{k}\left(\left(\frac{l+\alpha -5}{l+\alpha -2}\right)^2-1\right)\|\Delta x_{l-1}\|^2.
    \end{equation} 
    Based on the \cref{lem:Lyapunovestimate}, we have  
    \begin{equation}
    \begin{aligned}
    \mathcal E(j)-\mathcal E(j-1)&\le -\frac{\alpha -3}{\alpha -1}\phi_j\|\Delta x_j\|^2+\frac{\alpha -3}{\alpha -1}(\phi_{j-1}-2)\|\Delta x_{j-1}\|^2\\
    &\quad +\ 4\frac{\alpha -3}{\alpha -1}\|\Delta x_{j-1}\|\|x_{j-1}-z\|-\frac{2s}{\alpha -1}\sigma_{j-1}(z)\\
    &\le -\frac{\alpha -3}{\alpha -1}\phi_j\|\Delta x_j\|^2+\frac{\alpha -3}{\alpha -1}\phi_{j-1}\|\Delta x_{j-1}\|^2-2\frac{\alpha -3}{\alpha -1}\|\Delta x_{j-1}\|^2\\
    &\quad +\ 4\frac{\alpha -3}{\alpha -1}\|\Delta x_{j-1}\|\|x_{j-1}-z\|\\&\quad +\frac{2s}{\alpha -1}\Bigg[-\sigma_{k}(z)+\frac1{2s}\Big(\|\Delta x_{j-1}\|^2-\|\Delta x_{k}\|^2\Big)\\&\quad +\frac1{2s}\sum_{l=j}^{k}\left(\left(\frac{l+\alpha -5}{l+\alpha -2}\right)^2-1\right)\|\Delta x_{l-1}\|^2\Bigg].\\
    \end{aligned}
    \end{equation}
    Let \(j_1 = 2\) and \(j_2 = k\). Then  
    \begin{equation}
    \begin{aligned}
    \mathcal E(k)-\mathcal E(1 )&\le -\frac{\alpha -3}{\alpha -1}\phi_k\|\Delta x_k\|^2+\frac{\alpha -3}{\alpha -1}\phi_{1}\|\Delta x_{1 }\|^2\\&\quad -2\frac{\alpha -3}{\alpha -1}\sum_{j= 2}^k\|\Delta x_{j-1}\|^2+4\frac{\alpha -3}{\alpha -1}\sum_{j=2}^k\|\Delta x_{j-1}\|\|x_{j-1}-z\|\\
    &\quad -\frac{2s}{\alpha -1}k\sigma_k(z)+\frac{1}{\alpha -1}\sum_{j= 2}^k\|\Delta x_{j-1}\|^2\\
    &\quad+ \frac1{\alpha -1}\sum_{j= 2}^k\sum_{l=j}^{k}\left(\left(\frac{l+\alpha -5}{l+\alpha -2}\right)^2-1\right)\|\Delta x_{l-1}\|^2\\
    &\le \frac{\alpha -3}{\alpha -1}\phi_{1}\|\Delta x_{1}\|^2+4\frac{\alpha -3}{\alpha -1}\sum_{j=2}^k \|\Delta x_{j-1}\|\|x_{j-1}-z\|\\
    &\quad -\frac{2s}{\alpha -1}k\sigma_k(z)-2\frac{\alpha -3}{\alpha -1}\sum_{l=2}^k\|\Delta x_{l-1}\|^2\\
    &\quad +\frac{1}{\alpha -1}\sum_{l= 2}^{k}\left(1+(l-1)\left(\left(\frac{l+\alpha -5}{l+\alpha -2}\right)^2-1\right)\right)\|\Delta x_{l-1}\|^2\\
    &\le \frac{\alpha -3}{\alpha -1}\phi_{1}\|\Delta x_{1}\|^2+4\frac{\alpha -3}{\alpha -1}\sum_{j=2}^k \|\Delta x_{j-1}\|\|x_{j-1}-z\|-\frac{2s}{\alpha -1}k\sigma_k(z)\\
    &\quad +\frac{1}{\alpha -1}\sum_{l= 2}^{k}\left(1+(l-1)\left(\left(\frac{l+\alpha -5}{l+\alpha -2}\right)^2-1\right)-2(\alpha -3)\right)\|\Delta x_{l-1}\|^2.
    \end{aligned}
    \end{equation}  
    Let the function  
    \begin{equation} 
    h(l)=1+(l-1)\left(\left(\frac{l+\alpha -5}{l+\alpha -2}\right)^2-1\right)
    \end{equation}
    Then its derivative is  
    \begin{equation}
    \begin{aligned}
    h'(l)
    &=\frac{-3(2\alpha+1)l -6\alpha^2 +27\alpha -12}{(l+\alpha -2)^3} 
    \end{aligned}
    \end{equation}
    Since the numerator is a linear function and the coefficient $-3(2\alpha+1)<0$, we have  
    \begin{equation}\label{eq:HH}
    h'(l)\le h'(2)=\frac{ -6\alpha^2 +15\alpha -18}{(l+\alpha -2)^3}\le 0
    \end{equation}
    The latter inequality holds because $\alpha \ge 3$. Furthermore,  \cref{eq:HH} implies that $h(l)$ is monotonically decreasing on $[2,+\infty )$, and thus  
    $$
    h(l)\le h(2)=\frac{(\alpha -3)^2}{\alpha^2}\le 2(\alpha -3)
    $$
    for any $\alpha \ge 3$.
    It is easy to see that  
    \begin{equation}
    1+(l-1)\left(\left(\frac{l+\alpha -5}{l+\alpha -2}\right)^2-1\right)<2(\alpha -3),  
    \end{equation} 
    { for all $ l \ge 2$}.
    Therefore, we have  
    \begin{equation}
    \mathcal E(k)-\mathcal E(1 )\le \frac{\alpha -3}{\alpha -1}\phi_{1 }\|\Delta x_{1}\|^2+4\frac{\alpha -3}{\alpha -1}\sum_{j=2}^k \|\Delta x_{j-1}\|\|x_{j-1}-z\| -\frac{2s}{\alpha -1}k\sigma_k(z).  
    \end{equation}  The proof is complete.\qed 
\section{Test problem}\label{appendix:test_Problem}
\subsection{Test problems for ODE}\label{section:test_ODE}
\textbf{Quadratic probelm}.  As in \cite{sonntag2024fastgradientflow}, define the following functions:
\[
f_1 : \mathbb{R}^2 \to \mathbb{R}, \quad x = (x_1, x_2)^\top  \mapsto (x_1 - 1)^2 + \frac{1}{2}x_2^2,
\]
\[
f_2 : \mathbb{R}^2 \to \mathbb{R}, \quad x = (x_1, x_2)^\top  \mapsto \frac{1}{2}x_1^2 + (x_2 - 1)^2.
\]
The Pareto solution set for \(\min_{x \in \mathbb{R}^2} (f_1, f_2)\) is:
\[
P = \left\{ x \in \mathbb{R}^2 \,\middle|\, x = 
\begin{bmatrix}
\frac{2\lambda}{1 + \lambda} \\[6pt]
\frac{2(1 - \lambda)}{2 - \lambda}
\end{bmatrix}, \quad \lambda \in [0, 1] \right\}.
\]
Use initial point \(x_0 = (-0.2,- 0.1)^\top \). 

\noindent\textbf{Non-Quadratic problem}. Consider the following two functions, also discussed in \cite{sonntag2024fastgradientflow}:
\[
f_1 : \mathbb{R}^2 \to \mathbb{R}, \quad x = (x_1, x_2)^\top  \mapsto \ln \left( \sum_{j=1}^{4} \exp \left( a_j^\top  x - b_j \right) \right),
\]
\[
f_2 : \mathbb{R}^2 \to \mathbb{R}, \quad x = (x_1, x_2)^\top  \mapsto \ln \left( \sum_{j=1}^{4} \exp \left( a_j^\top  x + b_j \right) \right),
\]
where
\[
\begin{bmatrix}
a_1^\top  \\
a_2^\top  \\
a_3^\top  \\
a_4^\top 
\end{bmatrix}
=
\begin{bmatrix}
10 & 10 \\
10 & -10 \\
-10 & -10 \\
-10 & 10
\end{bmatrix},
\quad 
\begin{bmatrix}
b_1 \\
b_2 \\
b_3 \\
b_4
\end{bmatrix}
=
\begin{bmatrix}
0 \\
-20 \\
0 \\
20
\end{bmatrix}.
\]
and the Pareto set is given by:
\[
P = \left\{ x \in \mathbb{R}^2 \,\middle|\, x = \begin{pmatrix}
-1 + 2\lambda \\
1 - 2\lambda
\end{pmatrix}, \quad \lambda \in [0, 1] \right\}.
\]

\subsection{Test problems for algorithm}
\begin{itemize}
\item[1.] \textbf{JOS1}, quadratic programming  
$$
\begin{aligned}
f_1(x)&=\frac1{n}\sum_{i=1}^m x_i^2,\\
f_2(x)&=\frac1n\sum_{i=1}^m(x_i-2)^2,
\end{aligned}
$$  

\item[2. ]\textbf{FDS}, convex problem  
$$
\begin{aligned}
f_1(x)&=\frac1{n^2}\sum_{i=1}^n i(x_i-i)^4,\\
f_2(x)&=\exp\left(\sum_{i=1}^n\frac{x_i}{n}\right)+\|x\|^2,\\
f_3(x)&=\frac{1}{n(n+1)}\sum_{i=1}^ni(n-i+1)e^{-x_i}
\end{aligned}
$$  

\item[3. ]\textbf{LTY1}, convex problem.  
$$
f_j(x)=\ln\sum_{i=1}^p\exp\left(\langle a_i^j,x\rangle -b_i^j\right)+\frac\delta2\|x\|^2,\quad j=1,2,3
$$  
where $\delta =0.05$, $p=100$, $a_i^j\in \R^n$, and $b_i^j\in \R$ for $1\le i\le p$, with each component uniformly selected from $[-1,1]$.  

 \item[4. ]\textbf{LTY2}, convex problem.  
$$
f_j(x)=\frac12\|A^j x-b^j\|^2+\frac{\delta }{2}\|x\|^2,\quad j=1,2,3
$$  
where $\delta =0.05$, $p=5$, $A^j \in \R^{n\times p}$, and $b^j\in \R^p$, with each component uniformly selected from $[-1,1]$.  

\item[5. ] \textbf{DD1}, non-convex problem  
$$
\begin{aligned}
F_1(x)&=x_1^2+x_2^2+x_3^2+x_4^2+x_5^2\\
F_2(x)&=3x_1+2x_2-\frac{x_3}{3}+0.01(x_4-x_5)^3
\end{aligned}
$$  

\item[6. ] \textbf{KW2}, non-convex problem.  
$$
\begin{aligned}
f_1(x)&=-3(1-x_1)^2\exp(-x_1^2-(x_2+1)^2)\\
&\quad + 10\left(\frac{x_1}{5}-x_1^3-x_2^5\right)\exp(-x_1^2-x_2^2)\\
&\quad +3\exp(-(x_1+2)^2-x_2^2)-0.5(2x_1+x_2),\\
f_2(x)&=-3(1+x_2)^2\exp(-x_2^2-(1-x_1)^2)\\
&\quad +10\left(-\frac{x_2}{5}+x_2^3+x_1^5\right)\exp(-x_1^2-x_2^2)\\
&\quad +3\exp(-(2-x_2)^2-x_1^2),
\end{aligned}
$$

\item[7. ]  \textbf{LTY3}, non-convex problem.  
$$
\begin{aligned}
f_1(x)&=\frac12\left(\sqrt{1+|a_1^\top x|^2}+\sqrt{1+|a_2^\top x|^2}+a_2^\top x\right)+\exp(-|a_2^\top x|^2),\\
f_2(x)&=\frac12\left(\sqrt{1+|a_1^\top x|^2}+\sqrt{1+|a_2^\top x|^2}-a_2^\top x\right)+\exp(-|a_2^\top x|^2),
\end{aligned}
$$  
where $a_1$ and $a_2$ are uniformly selected from $[0,1]^n$.  

\item[8. ]  \textbf{SD}, convex problem\footnote{In \cite{mita2019nonmonotone}, Fukuda et al. provided the case for $n=4$, and here we have generalized it.\label{foot:noteSD}}, $x\in \R^n$.  
$$
\begin{aligned}
f_1(x)&=2x_1+\sqrt 2\sum_{i=2}^{n-1}{x_i}+x_n\\
f_2(x)&=\frac{2}{x_1}+2\sqrt 2\sum_{i=2}^{n-1}\frac{1}{x_i}+\frac{2}{x_n}
\end{aligned}
$$  

\item[9. ]\textbf{TOI4}, convex problem\textsuperscript{\ref{foot:noteSD}}, $x\in \R^n$, $n=2k$. 
$$
\begin{aligned}
f_1(x)&=\sum_{i=1}^kx_i^2+1\\
f_2(x)&=0.5\left(\sum_{i=1}^k(x_{2i-1}-x_{2i})^2\right)+1
\end{aligned}
$$
\end{itemize}
\end{appendices}



\begin{thebibliography}{spmpsci}



\bibitem{Aliprantis2006Infinite}
C.~D. Aliprantis and K.~C. Border.
\newblock {\em Infinite Dimensional Analysis: A Hitchhiker's Guide}.
\newblock Springer Science \& Business Media, Berlin, Heidelberg, 3 edition,
2006.

\bibitem{attouch2022damped}
H.~Attouch, A.~Balhag, Z.~Chbani, and H.~Riahi.
\newblock Damped inertial dynamics with vanishing tikhonov regularization:
Strong asymptotic convergence towards the minimum norm solution.
\newblock {\em Journal of differential equations}, 311:29--58, 2022.

\bibitem{attouch2018fast}
H.~Attouch, Z.~Chbani, J.~Peypouquet, and P.~Redont.
\newblock Fast convergence of inertial dynamics and algorithms with asymptotic
vanishing viscosity.
\newblock {\em Mathematical Programming}, 168:123--175, 2018.

\bibitem{attouch2019fast}
H.~Attouch, Z.~Chbani, and H.~Riahi.
\newblock Fast convex optimization via time scaling of damped inertial gradient
dynamics.
\newblock 2019.

\bibitem{attouch2015multiibjective}
H.~Attouch and G.~Garrigos.
\newblock Multiibjective optimization: an inertial dynamical approach to pareto
optima.
\newblock {\em arXiv preprint arXiv:1506.02823}, 2015.

\bibitem{Attouch2014}
H.~Attouch and X.~Goudou.
\newblock A continuous gradient-like dynamical approach to pareto-optimization
in hilbert spaces.
\newblock {\em Set-Valued and Variational Analysis}, 22:189--219, 2014.

\bibitem{attouch2024convex}
H.~Attouch and S.~C. L{\'a}szl{\'o}.
\newblock Convex optimization via inertial algorithms with vanishing tikhonov
regularization: fast convergence to the minimum norm solution.
\newblock {\em Mathematical Methods of Operations Research}, 99(3):307--347,
2024.

\bibitem{aubin2009differential}
J.-P. Aubin, H.~Frankowska, J.-P. Aubin, and H.~Frankowska.
\newblock {\em Differential inclusions}.
\newblock Springer, 2009.

\bibitem{bertsekas2009convex}
D.~Bertsekas.
\newblock {\em Convex optimization theory}, volume~1.
\newblock Athena Scientific, 2009.

\bibitem{bot2024inertial}
R.~I. Bo{\c{t}} and K.~Sonntag.
\newblock Inertial dynamics with vanishing tikhonov regularization for
multiobjective optimization.
\newblock {\em Journal of Mathematical Analysis and Applications}, page 129940,
2025.

\bibitem{burachik2017new}
R.~S. Burachik, C.~Y. Kaya, and M.~Rizvi.
\newblock A new scalarization technique and new algorithms to generate pareto
fronts.
\newblock {\em SIAM Journal on Optimization}, 27(2):1010--1034, 2017.

\bibitem{dontchev2009implicit}
A.~L. Dontchev and R.~T. Rockafellar.
\newblock {\em Implicit functions and solution mappings}, volume 543.
\newblock Springer, 2009.

\bibitem{fliege2000steepest}
J.~Fliege and B.~F. Svaiter.
\newblock Steepest descent methods for multicriteria optimization.
\newblock {\em Mathematical methods of operations research}, 51:479--494, 2000.

\bibitem{luc1989vector}
D.~T. Luc.
\newblock {\em Theory of Vector Optimization}, volume 319 of {\em Lecture Notes
    in Economics and Mathematical Systems}.
\newblock Springer-Verlag, Berlin, Heidelberg, 1989.

\bibitem{luo2023accelerated}
H.~Luo.
\newblock Accelerated differential inclusion for convex optimization.
\newblock {\em Optimization}, 72(5):1139--1170, 2023.

\bibitem{luo2022differential}
H.~Luo and L.~Chen.
\newblock From differential equation solvers to accelerated first-order methods
for convex optimization.
\newblock {\em Mathematical Programming}, 195(1):735--781, 2022.

\bibitem{luo2025accelerated}
H.~Luo, L.~Tang, and X.~Yang.
\newblock An accelerated gradient method with adaptive restart for convex
multiobjective optimization problems.
\newblock {\em arXiv preprint arXiv:2501.07863}, 2025.

\bibitem{miettinen1999nonlinear}
K.~Miettinen.
\newblock {\em Nonlinear multiobjective optimization}, volume~12.
\newblock Springer Science \& Business Media, 1999.

\bibitem{mita2019nonmonotone}
K.~Mita, E.~H. Fukuda, and N.~Yamashita.
\newblock Nonmonotone line searches for unconstrained multiobjective
optimization problems.
\newblock {\em Journal of Global Optimization}, 75(1):63--90, 2019.

\bibitem{sonntag2024fastgradientflow}
K.~Sonntag and S.~Peitz.
\newblock Fast convergence of inertial multiobjective gradient-like systems
with asymptotic vanishing damping.
\newblock {\em SIAM Journal on Optimization}, 34(3):2259--2286, 2024.

\bibitem{sonntag2024fastNestrovAlgorithm}
K.~Sonntag and S.~Peitz.
\newblock Fast multiobjective gradient methods with nesterov acceleration via
inertial gradient-like systems.
\newblock {\em Journal of Optimization Theory and Applications},
201(2):539--582, 2024.

\bibitem{su2016differential}
W.~Su, S.~Boyd, and E.~J. Candes.
\newblock A differential equation for modeling nesterov's accelerated gradient
method: Theory and insights.
\newblock {\em Journal of Machine Learning Research}, 17(153):1--43, 2016.

\bibitem{tanabe2019proximal}
H.~Tanabe, E.~H. Fukuda, and N.~Yamashita.
\newblock Proximal gradient methods for multiobjective optimization and their
applications.
\newblock {\em Computational Optimization and Applications}, 72:339--361, 2019.

\bibitem{tanabe2022globally}
H.~Tanabe, E.~H. Fukuda, and N.~Yamashita.
\newblock A globally convergent fast iterative shrinkage-thresholding algorithm
with a new momentum factor for single and multiobjective convex
optimization.
\newblock {\em arXiv preprint arXiv:2205.05262}, 2022.

\bibitem{tanabe2023accelerated}
H.~Tanabe, E.~H. Fukuda, and N.~Yamashita.
\newblock An accelerated proximal gradient method for multiobjective
optimization.
\newblock {\em Computational Optimization and Applications}, 86(2):421--455,
2023.

\bibitem{tanabe2023convergence}
H.~Tanabe, E.~H. Fukuda, and N.~Yamashita.
\newblock Convergence rates analysis of a multiobjective proximal gradient
method.
\newblock {\em Optimization Letters}, 17(2):333--350, 2023.

\bibitem{tanabe2024new}
H.~Tanabe, E.~H. Fukuda, and N.~Yamashita.
\newblock New merit functions for multiobjective optimization and their
properties.
\newblock {\em Optimization}, 73(13):3821--3858, 2024.

\bibitem{wang2021search}
Y.~Wang, Z.~Jia, and Z.~Wen.
\newblock Search direction correction with normalized gradient makes
first-order methods faster.
\newblock {\em SIAM Journal on Scientific Computing}, 43(5):A3184--A3211, 2021.

\bibitem{yang2024global}
Y.~Yang.
\newblock A global barzilai and borwein's gradient normalization descent method
for multiobjective optimization.
\newblock {\em arXiv preprint arXiv:2403.05070}, 2024.

\bibitem{yin2025multiobjective}
Y.~Yin.
\newblock Multiobjective balanced gradient flow.
\newblock {\em arXiv preprint arXiv:2508.01775}, 2025.

\bibitem{yin2025multiobjective2}
Y.~Yin.
\newblock Time scaling makes accelerated gradient flow and proximal method faster in multiobjective optimization.
\newblock {\em arXiv preprint arXiv:2508.07254}, 2025.
\end{thebibliography}
\end{document}